\documentclass[reqno]{article}
\usepackage{amsmath,amsfonts,amssymb,amsthm,color}

\newtheorem{theorem}{Theorem}
\newtheorem{lemma}[theorem]{Lemma}
\newtheorem*{claim*}{Claim}
\newtheorem{claim}[theorem]{Claim}
\newtheorem{corollary}[theorem]{Corollary}

\theoremstyle{definition}
\newtheorem{definition}[theorem]{Definition}
\newtheorem{remark}[theorem]{Remark}
\newtheorem*{remark*}{Remark}

\newtheorem{assumption1}[theorem]{Standard Assumption}
\newtheorem{assumption2}[theorem]{Strong Assumption}

\renewcommand{\le}{\leqslant}
\renewcommand{\ge}{\geqslant}

\newcommand\noproof{\hfill$\Box$}

\newcommand\Bi{{\mathrm{Bin}}}
\newcommand\eps{\varepsilon}
\newcommand\la{\lambda}
\newcommand\La{\Lambda}
\newcommand\las{\lambda_*}

\renewcommand\Pr{{\mathbb P}}
\newcommand\E{{\mathbb E}}
\newcommand\Var{{\mathrm{Var}}}
\newcommand\Covar{{\mathrm{Cov}}}
\newcommand\dto{\overset{\mathrm{d}}{\to}}
\newcommand\pto{\overset{\mathrm{p}}{\to}}
\newcommand\cc{{\mathrm{c}}}
\newcommand\op{o_{\mathrm{p}}}
\newcommand\Op{O_{\mathrm{p}}}
\newcommand\cF{\mathcal{F}}

\newcommand\cA{\mathcal{A}}
\newcommand\cB{\mathcal{B}}
\newcommand\cC{\mathcal{C}}

\newcommand\cL{\mathcal{L}}

\newcommand\cU{\mathcal{U}}

\newcommand\cE{\mathcal{E}}

\newcommand\tX{\widetilde{X}}

\newcommand\bb[1]{\bigl(#1\bigr)}
\newcommand\Bb[1]{\Bigl(#1\Bigr)}

\newcommand\ind[1]{1_{#1}}

\newcommand\dd{\mathrm{d}}

\newcommand\Hrnp{H^r_{n,p}}

\newcommand\Htnp{H^2_{n,p}}

\begin{document}
\title{Exploring hypergraphs with martingales}
\author{B\'ela Bollob\'as%
\thanks{Department of Pure Mathematics and Mathematical Statistics,
Wilberforce Road, Cambridge CB3 0WB, UK and
Department of Mathematical Sciences, University of Memphis, Memphis TN 38152, USA.
E-mail: {\tt b.bollobas@dpmms.cam.ac.uk}.}
\thanks{Research supported in part by NSF grant DMS-1301614 and
EU MULTIPLEX grant 317532.}
\and Oliver Riordan%
\thanks{Mathematical Institute, University of Oxford, Radcliffe Observatory Quarter, Woodstock Road, Oxford OX2 6GG, UK.
E-mail: {\tt riordan@maths.ox.ac.uk}.}}
\date{March 25, 2014; revised February 2, 2016}
\maketitle

\begin{abstract}
Recently, in~\cite{BR_hyp} we adapted exploration and martingale arguments
of Nachmias and Peres~\cite{NP_giant}, in turn
based on ideas of Martin-L\"of~\cite{ML86}, Karp~\cite{Karp} and Aldous~\cite{Aldous},
to prove asymptotic normality
of the number $L_1$ of vertices in the largest component $\cL_1$ of the random $r$-uniform
hypergraph in the supercritical regime.
In this paper we take these
arguments further to prove two new results: strong tail bounds on the distribution
of $L_1$, and joint asymptotic normality of $L_1$ and the number $M_1$ of edges
of $\cL_1$ in the sparsely supercritical case. These results are used in~\cite{BRsmoothing},
where we enumerate sparsely connected hypergraphs asymptotically.
\end{abstract}

\section{Introduction and results}

For $2\le r\le n$ and $0<p<1$, let $\Hrnp$ denote the random $r$-uniform hypergraph
with vertex set $[n]=\{1,2,\ldots,n\}$ in which
each of the $\binom{n}{r}$ possible hyperedges is present independently
with probability $p$. One family of interesting questions concerning $\Hrnp$ asks for
analogues of the pioneering results of Erd\H os and R\'enyi~\cite{ERgiant}
concerning the phase transition in the graph ($r=2$) case of this model,
as well as analogues of the many more detailed and precise results that followed.
Throughout the paper we fix $r\ge 2$ and consider 
\[
 p=p(n)=\la (r-2)!n^{-r+1}
\]
with $\la=\la(n)=\Theta(1)$.
The reason for this normalization is that,
as shown by Schmidt-Pruzan and Shamir~\cite{S-PS}, with
this choice $\la=1$ is the critical point
of the phase transition in $\Hrnp$, above which a giant component emerges.

For $r=2$, a great deal is known; for $r\ge 3$, most past results
concern the case $\la\ne 1$ constant, or (essentially equivalently), $\la=1\pm\Theta(1)$.\footnote{Given functions $f(n)$ and $g(n)$ with $g(n)>0$ for $n\ge n_0$,
we write $f(n)=O(g(n))$ if $\limsup_{n\to\infty} |f(n)|/g(n)<\infty$, i.e., there is a constant $C>0$
such that $|f(n)|\le Cg(n)$ for all $n\ge n_0$. We write $f(n)=\Theta(g(n))$ if there are \emph{positive} constants
$C>c>0$ such that $cg(n)\le f(n)\le Cg(n)$ for all large enough $n$.
Similarly, $f(n)=\Omega(g(n))$ if $\exists n_0,c>0$ such that $f(n)\ge cg(n)$ for $n\ge n_0$.}
Here we are especially
interested in what happens when $\la\to 1$, so much of the time we write
$\la=1+\eps$ or $\la=1-\eps$, with $\eps=\eps(n)\to 0$. In~\cite{BR_hyp},
a result of Aldous~\cite{Aldous} concerning critical random graphs ($r=2$) is
extended to $r\ge 3$; this implies in particular that
the critical window of the phase transition in $\Hrnp$
is when $\eps^3n=O(1)$, just as in the graph case.
Here we study $\Hrnp$ \emph{outside} the critical window, i.e., when $\eps^3n\to\infty$.

If $G$ is a (multi-)graph, then its \emph{nullity} is
\[ 
 n(G) = c(G) +e(G) - |G|,
\]
where $|G|$, $e(G)$ and $c(G)$ are the numbers of vertices, edges and components of $G$.
In the hypergraph case, it is natural to define the nullity of $H$
as the nullity of any multigraph obtained by replacing each hyperedge by a tree
on the same set of vertices. In the $r$-uniform case, this reduces to the following definition:
\[
 n(H) = c(H)+(r-1)e(H)-|H|.
\]
For connected graphs and hypergraphs, one often studies instead the \emph{excess} $n(G)-1$
or $n(H)-1$. However, while this definition is natural for connected graphs (where it reduces to $e(G)-|G|$),
it seems less natural for hypergraphs, and we prefer to work with $n(H)$.

Let $\cL_1$ be the component of $\Hrnp$
containing the most vertices, chosen according to any rule if there is a tie.
Let $L_1=|\cL_1|$ and $M_1=e(\cL_1)$ be the numbers of vertices and edges in $\cL_1$,
and $N_1=n(\cL_1)$ its nullity, so
\[
 (r-1)M_1 = L_1 +N_1 -1.
\]
Our main aim is to prove a bivariate central limit theorem (Theorem~\ref{thglobal2} below)
for the random variable $(L_1,N_1)$ (and hence for $(L_1,M_1)$ and for $(M_1,N_1)$)
throughout the sparsely supercritical regime, i.e.,
when $\la=1+\eps$ with $\eps^3n\to\infty$ and $\eps\to 0$. The corresponding result for $\eps=\Theta(1)$
was proved recently by Behrisch, Coja-Oghlan and Kang~\cite{BC-OK2a},
as part of a stronger result, a local limit theorem. Their methods are completely
different from ours, and seem very unlikely to adapt to the case $\eps\to 0$.

Our second aim is to prove, in Theorems~\ref{thsubtail} and~\ref{thsupertail} below,
large-deviation bounds on $L_1$ in the supercritical and subcritical cases.
As far as we are aware, even for $\eps=\Theta(1)$ these results are new
for hypergraphs, so here we do not assume that $\eps\to 0$.
As we show in a separate paper~\cite{BRsmoothing}, it is possible to use `smoothing'
arguments to deduce from Theorem~\ref{thglobal2} its local limit analogue,
and hence to give an asymptotic formula for the number of connected $r$-uniform
hypergraphs with $s$ vertices and $m$ edges, for suitable $m=m(s)$.
The tail bounds proved here are needed for these arguments
as well as being (we hope) of interest
in their own right.

To state our results precisely we need a number of definitions;
we shall (mostly) follow the notation in~\cite{BR_hyp}.
For $\la>1$ let $\rho_{\la}$ be the unique positive solution to
\begin{equation}\label{rldef}
 1-\rho_{\la} = e^{-\la \rho_\la},
\end{equation}
so $\rho_\la$ is the survival probability of a Galton--Watson branching process
whose offspring distribution is Poisson with mean $\la$,
and define $\las<1$, the parameter \emph{dual} to $\la$, by
\[
 \las e^{-\las} = \la e^{-\la}.
\]
It is easy to check that
\begin{equation}\label{las}
 \las=\la(1-\rho_\la),
\end{equation}
and that for any $A>1$ there exist $C>c>0$ such that
$\la=1+\eps\in (1,A]$ implies
\begin{equation}\label{lasrate2}
  1-C\eps \le \las \le 1-c\eps.
\end{equation}

For $\la>1$ and $r\ge 2$, define $\rho_{r,\la}$ by
\begin{equation}\label{rkldef}
 1-\rho_{r,\la} = (1-\rho_{\la})^{1/(r-1)},
\end{equation}
and set
\begin{equation}\label{rhosdef}
 \rho_{r,\la}^* = \frac{\la}{r}\bb{1-(1-\rho_{r,\la})^r} - \rho_{r,\la}.
\end{equation}
(The star here does not refer to duality; rather it is a notational convention
adopted from~\cite{norm2}.)
If $\la=1+\eps$ then, as $\eps\to 0$, elementary but tedious calculations show that
\begin{equation}\label{rsasymp}
 1-\las\sim \eps,\quad  \rho_{r,\la} \sim \frac{2\eps}{r-1}, \hbox{\quad and\quad} 
 \rho_{r,\la}^* \sim \frac{2}{3(r-1)^2}\eps^3.
\end{equation}
One way to see this is to use \eqref{rldef} to find (term-by-term) the first few
terms in a series expansion for $\rho_\la$, and to substitute this expansion
into~\eqref{rkldef} and then~\eqref{rhosdef}.\footnote{It turns out that with $\la=1+\eps>1$ we have
$\rho_\la = 2\eps - \frac{8}{3}\eps^2 + O(\eps^3)$.
This gives
\[
 \rho_{r,\la} = \frac{2}{r-1}\eps - \frac{2(r+2)}{3(r-1)^2}\eps^2 + O(\eps^3),
\]
which is enough to establish \eqref{rsasymp}.}

In~\cite{BR_hyp} we showed that
throughout the supercritical regime, i.e., when $\eps^3n\to\infty$ and $\eps=O(1)$,
the random variable $L_1(\Hrnp)$ is asymptotically
normally distributed with mean $\rho_{r,\la}n$ and variance $\sigma_{r,\la}^2n$,
where a formula for $\sigma_{r,\la}$ is given in~\cite[Eq. (3)]{BR_hyp}.
As noted there, when $\eps\to 0$, $\sigma_{r,1+\eps}^2\sim 2\eps^{-1}$.
Hence, under this additional assumption, the main result of~\cite{BR_hyp}
says exactly that $L_1(\Hrnp)$ is
asymptotically normally distributed with mean $\rho_{r,\la}n$ and variance $2n/\eps$.
Our first result extends this univariate central limit theorem to a bivariate one.

\begin{theorem}\label{thglobal2}
Let $r\ge 2$ be fixed, and let $p=p(n)=(1+\eps)(r-2)!n^{-r+1}$ where $\eps=\eps(n)\to 0$
and $\eps^3n\to\infty$. Let $L_1$ and $N_1$ be the order and nullity of the largest
component $\cL_1$ of $\Hrnp$. Then
\[
 \left( \frac{L_1-\rho_{r,\la}n}{\sqrt{2n/\eps}} ,
  \frac{N_1-\rho_{r,\la}^*n}{\sqrt{10/3}(r-1)^{-1}\sqrt{\eps^3 n}} \right)
\dto (Z_1,Z_2)
\]
as $n\to\infty$,
where $\rho_{r,\la}$ and $\rho_{r,\la^*}$ are defined as in \eqref{rkldef} and \eqref{rhosdef}
with $\la=1+\eps$,
$\dto$ denotes convergence in distribution, and $(Z_1,Z_2)$
has a bivariate Gaussian distribution
with mean $0$, $\Var[Z_1]=\Var[Z_2]=1$ and $\Covar[Z_1,Z_2]=\sqrt{3/5}$.
\end{theorem}

The graph case of this result was proved by Pittel and Wormald~\cite{PWio} using
very different methods, as part
of a stronger result. As noted above, the corresponding result with $\eps=\Theta(1)$
was proved recently by Behrisch, Coja-Oghlan and Kang~\cite{BC-OK2a}.
Their formula for the quantity corresponding to $\rho_{r,\la}^*$
coincides with ours, though the different notation obscures this. (They write
$\rho$ for $1-\rho$, and study $M_1$ rather than $N_1$. Since $M_1=(L_1+N_1-1)/(r-1)$, it is
straightforward to translate.)
We believe that our proof of Theorem~\ref{thglobal2} can be made to work
replacing the assumption $\eps\to 0$ by $\eps=O(1)$,
but the calculations would be more involved.
Since the result for $\eps=\Theta(1)$ is covered by that in~\cite{BC-OK2a},
we assume that $\eps\to 0$ to keep things simple.

\smallskip
We next turn to 
tail bounds on the distribution of $L_1(\Hrnp)$ in the subcritical and supercritical cases.
In reading these results, it is worth noting that in both cases, for deviations of order
$\eps n$, i.e., of order the typical value of $L_1(\Hrnp)$ in the supercritical case,
we obtain a bound on the probability of order $\exp(-\Omega(\eps^3n))$.
This formula, which we believe to be tight up to the constant,
corresponds to the function $\exp(-\Omega(n))$ that one expects when $\la\ne 1$ is constant.
We start with the subcritical case.

\begin{theorem}\label{thsubtail}
Let $r\ge 2$ be fixed and let $p=p(n)=(1-\eps)(r-2)!n^{-r+1}$ where $\eps^3n\to\infty$
and $1-\eps$ is bounded away from $0$.
If $L=L(n)$ satisfies $\eps^2L\to\infty$ and $L=O(\eps n)$, then there is a constant $C>0$ such that
\begin{equation}\label{L1bd}
 \Pr( L_1(\Hrnp) > L ) \le C \frac{\eps n}{L} \exp(-\eps^2L/C)
\end{equation}
for all large enough $n$.
\end{theorem}
\begin{remark}\label{Runif}
The formal statement is that for every $r\ge 2$ and every pair of functions
$p(n)$ and $L(n)$ satisfying the given conditions, there exist $C>0$ and $n_0$ such that \eqref{L1bd}
holds for all $n\ge n_0$. In other words, the constant $C$ is allowed to depend on the choice of $r\ge 2$, and of the functions
$p=p(n)$ and $L=L(n)$. This type of statement is convenient when it comes to the proof,
since we can just take $p(n)$ and $L(n)$ as given, and not worry about how $C$ depends on them.
However, as usual in such contexts, uniformity over suitable sets of choices
for $p(n)$ and $L(n)$ follows automatically. 
More precisely, given $r\ge 2$ and $A>0$, Theorem~\ref{thsubtail} implies that
there is a constant $C>0$, depending
only on $r$ and $A$, such that
\eqref{L1bd} holds whenever $1-\eps\ge 1/A$, $L\le A\eps n$, and $n$, $\eps^3n$
and $\eps^2L$ are large enough.\footnote{%
Suppose not. Then for each $k=1,2,\ldots$ we may find values $n_k$, $\eps_k$
and $L_k$ 
with $1-\eps_k\ge 1/A$ and $L_k\le A\eps_k n_k$ such that \eqref{L1bd} does not hold for these
values with $C=k$, with, in addition, $\min\{n_k,\eps_k^3n_k,\eps^2_kL_k\}\ge k$.
Passing to a subsequence we may assume that $(n_k)$ is strictly increasing. But now we have
partial functions $\eps(n)$ and $L(n)$ (which we may complete to functions) satisfying the
assumptions of Theorem~\ref{thsubtail}. So there should be some $C$ and $k_0$ such that \eqref{L1bd}
holds for this sequence, i.e., for all $(n_k,\eps_k,L_k)$, $k\ge k_0$.
Considering any $k>\max\{C,k_0\}$ now gives a contradiction.}
\end{remark}

Theorem~\ref{thsubtail} gives a meaningful bound (a bound on the probability that is less than 1)
only when $L$ is at least
some constant times $\log(\eps^3 n)/\eps^2$, which, as shown by Karo\'nski and \L uczak~\cite{KL_giant},
is the typical order of $L_1$.
For us, the most important case is that with $L=\Theta(\eps n)$.
We believe that, apart from the constant in the exponent, the bound given in Theorem~\ref{thsubtail}
is best possible for essentially the entire range to which it applies.

In the supercritical case, we show that $L_1$ is concentrated around its mean, and
that the number $L_2$ of vertices in the second-largest component is unlikely to be large.
\begin{theorem}\label{thsupertail}
Let $r\ge 2$ be fixed, let $p=p(n)=(1+\eps)(r-2)!n^{-r+1}$ where $\eps=O(1)$
and $\eps^3n\to\infty$, and define $\rho_{r,\la}$ as in \eqref{rkldef} with $\la=1+\eps$.
If $\omega=\omega(n)\to\infty$ and 
$\omega=O(\sqrt{\eps^3n})$ then
\begin{equation}\label{st1}
 \Pr\Bb{ |L_1(\Hrnp)-\rho_{r,\la} n| \ge \omega\sqrt{n/\eps} } = \exp(-\Omega(\omega^2)).
\end{equation}
Moreover, if $L=L(n)$ satisfies $\eps^2L\to\infty$ and $L=O(\eps n)$, then there exists $C>0$ such that
\begin{equation}\label{st2}
 \Pr( L_2(\Hrnp) > L ) \le C \frac{\eps n}{L} \exp(-\eps^2L/C)
\end{equation}
for all large enough $n$.
\end{theorem}
\begin{remark*}
Again, the constant $C$, and the implicit constant in the $\Omega(\cdot)$ notation in \eqref{st1},
may depend on the choice of the `input' parameters $r\ge 2$, $(p(n))$, $(L(n))$ and $(\omega(n))$.
\end{remark*}

Since $\rho_{r,\la}n=\Theta(\eps n)$, the bound \eqref{st1}
implies in particular that if $\delta=\delta(n)\le 1/2$, say,
and $\delta\sqrt{\eps^3n}\to\infty$, then there is a constant $c>0$ such that
\begin{equation}
 \Pr\Bb{ (1-\delta)\rho_{r,\la}n \le L_1(\Hrnp) \le (1+\delta)\rho_{r,\la}n  } \ge 1- \exp(-c\delta^2\eps^3 n)
\end{equation}
for $n$ large enough. As in Remark~\ref{Runif} above, one can check that this constant
depends only on $r$ and the implicit constant in our assumption $\eps=O(1)$.

For the largest component,
much more precise results are known in the graph case, at least when $\eps=\Theta(1)$:
for $L_1(\Htnp)$, $p=c/n$, O'Connell~\cite{OC} established a `large
deviation principle' tight up to a factor $1+o(1)$ in the exponent in the error probability.
Biskup, Chayes and Smith~\cite{BCS} proved a corresponding
result for the number of vertices in `large' components.

In the subcritical case,
Karo\'nski and \L uczak~\cite{KL_giant} proved
very precise results about the limiting distribution of $L_1$
(essentially a local limit result, but conditional on the probability
$1-o(1)$ event that there are no complex components). Theorem~\ref{thsubtail} neither implies their result
nor is implied by it: instead of considering `typical' values of $L_1$,  we prove that the probability that $L_1$ is considerably larger than such typical values goes to zero rather quickly.

The rest of the paper is organized as follows. We shall prove Theorems~\ref{thsubtail}, \ref{thsupertail}
and~\ref{thglobal2} in this order. First, in Section~\ref{sec_prelim},
we prove some simple lemmas that we shall need later. 
In Section~\ref{sec_expl}, we recall the exploration argument from~\cite{BR_hyp},
and state some basic properties of corresponding random walk.
In Section~\ref{sec_tail} we use this random walk to prove Theorem~\ref{thsubtail}.
Next, in section~\ref{sec_mart}, we describe the approximation of the random
walk by a martingale (as in~\cite{BR_hyp}). We use this to prove Theorem~\ref{thsupertail}
in Section~\ref{sec_sup} and our main result, Theorem~\ref{thglobal2},
in Section~\ref{sec_biv}.

\section{Preliminaries}\label{sec_prelim}

In this section we prove some probabilistic inequalities that will be needed
later. Here (and indeed throughout the paper) we
make no attempt to optimize the various constants that appear, or even to make them explicit.

\begin{lemma}\label{lbintail}
Let $k>0$. There is a constant $K=K(k)$ such that
if $Y\sim \Bi(n,p)$ with $np\le\nu\le k$, and $X$ is a non-negative random variable
with mean $\mu$ that is stochastically dominated by $kY$, then for
$-1\le\theta\le 1$ we have
\[
 \E[ (X-\mu)^2 e^{\theta (X-\mu) }] \le K\nu.
\]
\end{lemma}
\begin{proof}
For $0\le \alpha\le k$, by the binomial theorem and the standard inequality $1+x\le e^x$
we have
\begin{equation}\label{EaY}
 \E[e^{\alpha Y}]=  \sum_{k=0}^n \binom{n}{k}p^k(1-p)^{n-k}e^{\alpha k}
 = (1-p+p e^{\alpha})^n\le \exp(np(e^\alpha-1)) \le K_1,
\end{equation}
where $K_1=k(e^k-1)$ is a constant depending only on $k$.
Either by differentiating, or by using twice the observation that
$Y\sim \Bi(n,p)$ implies $\E[Yf(Y)]=np\E[f(Z+1)]$
where $Z\sim \Bi(n-1,p)$, we deduce that
\[
 \E[Y e^{\alpha Y}] \le np e^\alpha K_1 \le \nu e^k K_1
\]
and
\begin{equation}\label{EaY2}
 \E[Y^2 e^{\alpha Y}] = \E[Y(Y-1) e^{\alpha Y}] + \E[Ye^{\alpha Y}]
 \le  \nu^2e^{2k}K_1+\nu e^k K_1 \le \nu K_2,
\end{equation}
where $K_2=(ke^{2k}+e^k)K_1$. For $0\le\theta\le 1$,
since $\exp(\theta x)$ and $x^2\exp(\theta x)$
are increasing in $x\ge 0$, we have
\begin{multline}\label{EaX1}
 \E[ (X-\mu)^2 e^{\theta (X-\mu)} ] \le \E[(X-\mu)^2e^{\theta X}]
 \le \E[X^2e^{\theta X}]+\mu^2\E[e^{\theta X}] \\
 \le \E[k^2Y^2e^{\theta kY}]+\mu^2\E[e^{\theta kY}] \le k^2\nu K_2+\mu^2 K_1
 \le k^2\nu K_2+k\nu K_1 =\nu K_3,
\end{multline}
recalling \eqref{EaY} and \eqref{EaY2}, and noting that $\mu=\E[X]\le \E[Y]=np\le\nu\le k$.

Since $\mu\le k$ and $X\ge 0$,
for $-1\le \theta<0$ we have $e^{\theta(X-\mu)}\le e^{-\theta\mu}
\le e^k$, so
\[
 \E[ (X-\mu)^2 e^{\theta(X-\mu)} ] \le e^k\E[(X-\mu)^2] \le e^k \nu K_3,
\]
where in the last step we applied \eqref{EaX1} with $\theta=0$.
This completes the proof of the lemma with $K=e^kK_3$, a constant depending only on $k$.
\end{proof}

Our next lemma is a simple Hoeffding--Azuma-type martingale inequality that is doubtless a
special case of (many) known results. Since the proof is very simple, it seems
easiest just to give it. 

\begin{lemma}\label{lmart}
Let $C>0$ be a real number, and let
$(M_t)_{t=0}^\ell$ be a martingale with respect to the filtration
$(\cF_t)$ with $M_0=0$. Set $\Delta_t=M_t-M_{t-1}$, and suppose that for all
$1\le t\le \ell$ and all $\theta\in [-1,1]$ we have
\begin{equation}\label{Xbd}
 \E[\Delta_t^2 e^{\theta \Delta_t} \mid \cF_{t-1}] \le C \hbox{ almost surely.}
\end{equation}
Then
\begin{equation}\label{aimo}
 \Pr\Bb{ \max_{0\le t\le \ell} |M_t|\ge y } \le 
    2\exp\bb{-y^2/(2\max\{y,C\ell\}) }.
\end{equation}
\end{lemma}
\begin{proof}
By a standard stopping-time argument, to prove \eqref{aimo} it suffices to show that
\begin{equation}\label{aims}
 \Pr( |M_\ell|\ge y ) \le
    2\exp\bb{-y^2/(2\max\{y,C\ell\}) }.
\end{equation}
Indeed, let $\tau=\inf\{t:|M_t|\ge y\}\le\infty$ and consider the stopped martingale defined
by $M_t'=M_{t\wedge\tau}$. (Thus $M_t'=M_t$ for all $t$ if $\tau=\infty$.) This martingale also
satisfies the assumptions of the lemma, and relation \eqref{aims} for $(M_t')$ implies
\eqref{aimo} for $(M_t)$.

If $X$ is any random variable with $\E[X]=0$ satisfying $\E[X^2 e^{\theta X}] \le C$
for all $\theta\in [-1,1]$ then,
defining $f(\theta)=\E[e^{\theta X}]>0$, we have $f(0)=1$, $f'(0)=\E[X]=0$ and,
for $-1\le\theta\le 1$,
\[
 f''(\theta) = \E[X^2 e^{\theta X}] \le C.
\]
It follows that for $-1\le\theta\le 1$ we have
\[
 f(\theta) \le 1+ C\theta^2/2 \le \exp(C\theta^2/2).
\]

For $1\le t\le \ell$ let $\Delta_t=M_t-M_{t-1}$. Then $\E[\Delta_t\mid \cF_{t-1}]=0$
and, by assumption, for $-1\le\theta\le 1$ we have $\E[\Delta_t^2e^{\theta\Delta_t}\mid\cF_{t-1}] \le C$.
It follows that
\[
 \E[ e^{\theta\Delta_t} \mid \cF_{t-1}] \le \exp(C\theta^2/2).
\]
A standard inductive argument now implies that $\E[e^{\theta M_\ell}]\le \exp(C\theta^2\ell/2)$.
Let $y\ge 0$. Then, by Markov's inequality, for $0\le\theta\le 1$ we have
\[
 \Pr(M_\ell\ge y) \le \E[e^{\theta M_\ell}]/e^{\theta y} \le \exp(C\theta^2\ell/2-\theta y).
\]
For $y\le C\ell$, taking $\theta=y/(C\ell)\in [0,1]$ gives $\Pr(M_\ell\ge y)\le \exp(-y^2/(2C\ell))$;
for $y\ge C\ell$, taking $\theta=1$ gives $\Pr(M_\ell\ge y) \le \exp(C\ell/2-y)\le \exp(-y/2)$.
We may bound $\Pr(M_\ell\le -y)$ similarly, using Markov's inequality to show
that for $-1\le\theta\le 0$ we have
\[
 \Pr(M_\ell\le -y) \le \E[e^{\theta M_\ell}]/e^{-\theta y} \le \exp(C\theta^2\ell/2+\theta y),
\]
and then taking  $\theta=-y/(C\ell)$ or $\theta=-1$. This completes the proof \eqref{aims} and hence of the lemma.
\end{proof}

\section{The exploration process and its increments}\label{sec_expl}

Let us briefly recall some of the methods and results of~\cite{BR_hyp},
based on `exploring' the component structure of $\Hrnp$ step-by-step.\footnote{%
We aim for a presentation that is mostly self-contained: we shall need
some specific results from~\cite{BR_hyp} (see Lemmas~\ref{gprops}, \ref{lXXt} and \ref{props} and relation \eqref{T1tX} below), but hope that, taking these
on trust, it should be possible to follow the present paper without
reading~\cite{BR_hyp}. Having said this, there will be a few places
where we shall give a little less detail than we might otherwise have done, since further
detail is given in~\cite{BR_hyp}.}
Explorations of this type have been used on numerous occasions, including
by Martin-L\"of~\cite{ML86}, Karp~\cite{Karp}, Aldous~\cite{Aldous}
and Nachmias and Peres~\cite{NP_giant}. For hypergraphs, the form described
here was used by Behrisch, Coja-Oghlan and Kang~\cite{BC-OK1} and later
by the present authors in~\cite{BR_hyp}; in our opinion, the description
and analysis in~\cite{BR_hyp} is simpler than that in~\cite{BC-OK1}.
For further background, see~\cite{BR_walk}.

Given a hypergraph $H$
with vertex set $[n]$, we `explore' $H$ by revealing its edges in $n$ steps as follows.
In step $1\le t\le n$ we pick a vertex $v_t$ in a way that we shall specify in a moment,
and reveal all edges incident with $v_t$ but not with any of $v_1,\ldots,v_{t-1}$.
After $t$ steps we have `explored' the vertices $v_1,\ldots,v_t$,
and have revealed all edges incident with one or more of these vertices. An unexplored
vertex is `active' if it is incident with one or more revealed edges, and `unseen' otherwise.
We write $\cA_t$ for the set of active vertices after $t$ steps, $\cU_t$ for
the set of unseen vertices, and set $A_t=|\cA_t|$. 
When choosing which vertex to explore next, we pick an active vertex if there is one (according
to any rule), and an unseen vertex otherwise.

Let $0=t_0<t_1<t_2\cdots<t_\ell=n$ enumerate $\{t:A_t=0\}$. Then,
for $1\le i\le \ell$, the set $V_i=\{v_{t_{i-1}+1},\ldots,v_{t_i}\}$
is the vertex set of a component of $H$. Indeed,
for any $t$ such that $A_t=0$ there are no edges joining any $v_i$ with $i\le t$
to any $v_j$ with $j>t$, so $V_i$ is not joined to $[n]\setminus V_i$ in $H$,
and if $A_t>0$ then $v_{t+1}$ is active at time~$t$,
and hence is in some edge containing some $v_i$, $i\le t$; thus the subhypergraph
of $H$ induced by $V_i$ is connected. Hence, for $1\le i\le \ell$, $t_i$ is the step
at which we finish exploring the $i$th component of $H$.

Let
\[ 
 C_t=|\{i:0\le i<t,\, A_i=0\}|
\]
be the number of components that we have started to explore within
the first $t$ steps, and define $X_t=A_t-C_t$. As we shall see in a moment, the increments
of the process $(X_t)$ are simpler to understand that those of $(A_t)$, so, as in~\cite{BR_hyp},
we shall primarily study $(X_t)$. We can read off the component sizes from the trajectory
of $(X_t)$ without too much trouble. Indeed, since $C_t=1$ for $t=1,2,\ldots,t_1$,
we have $X_t=A_t-C_t\ge -1$ in this range with equality only at $t=t_1$.
Similarly, $X_t$ reaches a new `record low' value $-i$ at time $t_i$: 
\[
 t_i = \inf\{t:X_t=-i\}.
\]

Let $\eta_t$ be the number of vertices in $\cU_{t-1}\setminus\{v_t\}$ that become active in step $t$,
i.e., are contained in one or more hyperedges containing $v_t$ and none of $v_1,\ldots,v_{t-1}$.
In step $t$, exactly $\eta_t$ vertices become active. Moreover, either one vertex $v_t$
that was previously active ceases to be active, or we start a new component and so
$C_t=C_{t-1}+1$. In either case, $X_t-X_{t-1}=\eta_t-1$, so by induction
\begin{equation}\label{Xsum}
 X_t= \sum_{i=1}^t (\eta_i -1).
\end{equation}

So far, we have not specified the hypergraph $H$ that we are exploring. From now on, we take $H=\Hrnp$.
Let $\cF_t$ be the $\sigma$-algebra generated by all information revealed up to step $t$
of the exploration process.
This exploration process, the associated filtration $(\cF_t)$,
and the random sequences
$(X_t)$, $(\eta_t)$ and (to a lesser extent) $(A_t)$ and $(C_t)$ will be the tools
that we use throughout the paper to study $\Hrnp$.

We have not yet specified the function $p=p(n)$; we shall impose different assumptions
in different sections. But throughout the paper, we take $r\ge 2$ constant,
and assume that $p=p(n)=\Theta(n^{-r+1})$.

\begin{lemma}\label{cdist}
The distribution of $\eta_t$ conditional on $\cF_{t-1}$ is stochastically dominated by $r-1$
times a binomial random variable with mean
\[
 \binom{n-t}{r-1}p \le \binom{n}{r-1}p=O(1).
\]
\end{lemma}
\begin{proof}
In step $t$ we test exactly $\binom{n-t}{r-1}$ $r$-sets to see
whether they are edges of $H$, namely all $r$-sets including $v_t$ but none of $v_1,\ldots,v_{t-1}$.
None of these $r$-sets has been previously tested, so the random number $E_t$ of edges that we find has a binomial distribution
with mean $\binom{n-t}{r-1}p\le \binom{n}{r-1}p = O(1)$. The number $\eta_t$
of new active vertices is at most $(r-1)E_t$, with equality if and only if these edges
intersect only at $v_t$, and contain no previously active vertices other than $v_t$.
\end{proof}

In the rest of the paper we shall work with the Doob decomposition of the sequence $(X_t)$.
Set
\begin{align}
 D_t &= \E[\eta_t-1\mid \cF_{t-1}] \hbox{\quad and} \nonumber \\
 \Delta_t &= \eta_t-1-D_t = \eta_t-\E[\eta_t\mid \cF_{t-1}], \label{Dtdef}
\end{align}
so by definition $\E[\Delta_t\mid \cF_{t-1}]=0$ and, from \eqref{Xsum},
\[
 X_t= \sum_{i=1}^t (D_i+\Delta_i).
\]
Then $(\Delta_t)$ is by definition a martingale difference sequence with respect to the filtration $(\cF_t)$. We note two simple properties of the distribution of $\Delta_t$ which will be useful later.

\begin{lemma}\label{cvar}
Suppose that $p=p(n)=\la(n) (r-2)! n^{-r+1}$ with $\la(n)=\Theta(1)$.
Then there is a constant $C$ such that for all $n$
and all $1\le t\le n$ we have
\[
 \Var[\Delta_t\mid \cF_{t-1}] \le C
\]
with probability 1. Furthermore, if $t=t(n)=o(n)$ and $a=a(n)=o(n)$ then
\begin{equation}\label{Vsim}
 \Var[\Delta_t \mid \cF_{t-1}] \sim \la(r-1) \hbox{ whenever } A_{t-1}\le a.
\end{equation}
\end{lemma}
\begin{proof}
Condition on $\cF_{t-1}$. By Lemma~\ref{cdist}, the conditional distribution $X$ of $\eta_t$ is
stochastically dominated by $(r-1)Y$ where $Y\sim \Bi(\binom{n}{r-1},p)$.
Hence, writing $N=\binom{n}{r-1}$, we have
\begin{multline*}
 \Var[\Delta_t\mid\cF_{t-1}] = \Var[\eta_t\mid \cF_{t-1}] \le \E[\eta_t^2\mid \cF_{t-1}] \\
 = \E[X^2]  \le (r-1)^2\E[Y^2] = (r-1)^2\bb{N(N-1)p^2+Np}=O(1),
\end{multline*}
proving the first statement.

For the second, when $t=o(n)$ and we have $A_{t-1}=o(n)$ active vertices, it is easy to see that $X$
and $(r-1)Y$ are equal with probability $1-o(1)$. (The probability that any of the $(r-1)Y$
vertices are `duplicates' or lie in $A_{t-1}$ is $o(1)$.) This, together with stochastic
domination, implies that $\Var[X]\sim (r-1)^2\Var[Y]$. But $\Var[Y]$ is just $Np(1-p)\sim Np\sim \la/(r-1)$.
\end{proof}
Note that if $a(n)=o(n)$ then, by
by considering worst-case values, one can check that
the estimate \eqref{Vsim} holds uniformly over all $0\le t\le a(n)$ and all points
in the sample space at which $A_{t-1}\le a(n)$.

\section{The subcritical tail bound}\label{sec_tail}

In this section we prove the easiest of our main results, Theorem~\ref{thsubtail}; for this
we use simpler methods than those in~\cite{BR_hyp}. 

\begin{proof}[Proof of Theorem~\ref{thsubtail}]
Let $r\ge 2$ be fixed and let $p=p(n)=(1-\eps)(r-2)!n^{-r+1}$ where $\eps^3n\to\infty$
and $1-\eps$ is bounded away from $0$. Fix a function $L=L(n)$ satisfying
$\eps^2L\to\infty$ and $L=O(\eps n)$. Our aim is to show that for $n$ large enough we have
\[
 \Pr( L_1(\Hrnp) > L ) \le C \frac{\eps n}{L} \exp(-c \eps^2L),
\]
for some constants $c,C>0$ that may depend on all the choices made so far, just not (of course)
on $n$.

We explore the random hypergraph $\Hrnp$ as in Section~\ref{sec_expl}, defining the filtration
$(\cF_t)$ and random sequences $(X_t)$, $(\eta_t)$, $(A_t)$ and $(C_t)$
as in that section.
Recall also the definition \eqref{Dtdef} of $(D_t)$ and $(\Delta_t)$.
By Lemma~\ref{cdist},
\[
 \E[\eta_t\mid\cF_{t-1}] \le (r-1)\binom{n-t}{r-1}p \le \frac{(n-t)^{r-1}}{(r-2)!}p = (1-\eps)(1-t/n)^{r-1}.
\]
Let
\begin{equation}\label{atdef}
 a_t =  \eps +(1-\eps) t/n.
\end{equation}
Then, crudely,
\begin{equation}\label{Dt}
 D_t =\E[\eta_t-1\mid \cF_{t-1}] \le (1-\eps)(1-t/n)^{r-1}-1 
 \le (1-\eps)(1-t/n)-1 
 = -a_t.
\end{equation}
Note that $D_t$ is a random variable, but this deterministic bound holds with probability $1$.
Let
\[
 M_t = \sum_{i=1}^t \Delta_i = X_t -\sum_{i=1}^t D_i,
\]
so $(M_t)$ is a martingale with respect to $(\cF_t)$.  Since the sequence $(a_t)$
is increasing, from \eqref{Dt} we see that for $t_1<t_2$ we have
\begin{multline}\label{Mdiff}
 M_{t_2}-M_{t_1}
  = X_{t_2}-X_{t_1} - \sum_{t=t_1+1}^{t_2} D_t 
\ge X_{t_2}-X_{t_1} + \sum_{t=t_1+1}^{t_2} a_t \\
\ge X_{t_2}-X_{t_1} + (t_2-t_1)a_{t_1}.
\end{multline}

Suppose that $L_1(\Hrnp)>L$.
Then there is some $t$ (one less than the time at which we first start exploring a component with more than $L$
vertices) such that $A_t=0$, $A_{t+L}\ge 1$, and $C_{t+L}=C_{t+1}=C_t+1$. Thus
$X_{t+L}\ge X_t$.
For $j\ge 0$ let $\cE_j$ denote the event that there is a $t$
in the interval $jL\le t<(j+1)L$ with $X_{t+L}\ge X_t$. What we have just noted
tells us that
\[
 \Pr(L_1(\Hrnp)>L) \le \sum_{j=0}^\infty \Pr(\cE_j),
\]
so to complete the proof it suffices to bound the sum above.

If $\cE_j$ holds, then by definition there is a $t\in [jL,(j+1)L]$ such that $X_{t+L}\ge X_t$.
Then, by \eqref{Mdiff}, we have
\begin{equation}\label{Md2}
 M_{t+L}-M_t \ge L a_t \ge L a_{jL}.
\end{equation}
Consider the martingale $(M'_k)$ defined by $M_k'=M_{jL+k}-M_{jL}$, $k=0,\ldots,2L$.
If \eqref{Md2} holds then $M'_{t+L-jL}-M'_{t-jL}\ge La_{jL}$, so by the triangle inequality
$\max\{|M'_{t+L-jL}|,|M'_{t-jL}|\}\ge La_{jL}/2$.
Since $0\le t-jL\le L$, we find that if $\cE_j$ holds, then
\[
\max_{0\le k\le 2L}|M_k'|\ge L a_{jL}/2.
\]
By Lemmas~\ref{lbintail} and~\ref{cdist}, the martingale differences $\Delta_t=M_t'-M_{t-1}'$, $1\le t\le 2L$,
satisfy the hypothesis \eqref{Xbd} of Lemma~\ref{lmart} for some constant $C>0$.\footnote{In principle, as we have phrased the argument, $C$ may depend on the choice of $r\ge 2$ and also on the choice
of the function $p(n)$. Since we assume $p(n)\le (r-2)!n^{-r+1}$, it is not hard
to see that $C$ depends only on $r$.}
We may of course assume that $C\ge 1/4$. Then, by Lemma~\ref{lmart},
applied with $\ell=2L$ and $y=La_{jL}/2\le L/2\le 2CL=C\ell$, we have
\[
 \Pr(\cE_j) \le 2\exp\left(-\frac{y^2}{2C\ell}\right) = 2\exp\left(-\frac{L^2 a_{jL}^2/4}{4CL}\right)
 = 2\exp(-c a_{jL}^2L)
\]
where $c=1/(16C)$ is a positive constant.
From \eqref{atdef},
\[
 a_{jL}=\eps+(1-\eps)jL/n\ge \max\{\eps,(1-\eps)jL/n\}.
\]
Recalling that $1-\eps$ is bounded away from zero by assumption, and considering
$j<\eps n/L+1$ and $j\ge \eps n/L+1$ separately, it follows that
\begin{equation}\label{sjE}
 \sum_j \Pr(\cE_j) \le 2\left\lceil\frac{\eps n}{L}\right\rceil \exp(-c \eps^2L)
  + 2\sum_{j\ge \eps n/L+1} \exp(-c' j^2L^3/n^2),
\end{equation}
for some constant $c'>0$. Clearly
\begin{eqnarray*}
 \sum_{j\ge \eps n/L+1} \exp(-c' j^2L^3/n^2) &\le&
  \sum_{j\ge \eps n/L+1} \exp(-c' j\eps L^2/n) \\
 &\le& \exp(-c' \eps^2 L) \sum_{j\ge 1} \exp(-c' j\eps L^2/n),
\end{eqnarray*}
and if $x>0$ then
\[
 0< \sum_{j\ge 1} e^{-jx} = \frac{e^{-x}}{1-e^{-x}} = \frac{1}{e^x-1} < \frac{1}{x}.
\]
It follows that
\[
  \sum_{j\ge \eps n/L+1} \exp(-c' j^2L^3/n^2) = 
O\left(\frac{n}{\eps L^2}\right) \exp(-c' \eps^2 L).
\]
Finally, by assumption $\eps^2L\to\infty$, so $n/(\eps L^2) = o(\eps n/L)$ and,
from \eqref{sjE},
\[
 \sum_j \Pr(\cE_j) = O\left(\frac{\eps n}{L}\right) \exp\bb{-\min\{c,c'\} \eps^2 L},
\]
completing the proof of Theorem~\ref{thsubtail}.
\end{proof}

\section{Martingale approximation}\label{sec_mart}

In preparation for the proof of Theorem~\ref{thsupertail}, we recall and extend
some results from~\cite{BR_hyp}, approximating the random sequence
$(X_t)_{t=0}^n$ by the sum of a certain deterministic sequence and a martingale.

For the rest of the paper we make the following assumption.
\begin{assumption1}\label{A1}
The integer $r\ge 2$ is fixed, $\eps=\eps(n)$ is a function satisfying $\eps>0$, $\eps=O(1)$
and $\eps^3 n\to\infty$. Furthermore, $\la=\la(n)=1+\eps$ and $p=p(n)=\la (r-2)!n^{-r+1}$.
\end{assumption1}
As discussed in Remark~\ref{Runif}, all new constants introduced may depend on the choice of $r$
and of the function $\eps(n)$.

We start with some definitions, following the notation in~\cite{BR_hyp}.
Firstly, for $1\le t\le n$, set
\[
 \alpha_t= p \binom{n-t-1}{r-2}.
\]
Note that for all $t$ we have $0\le \alpha_t\le p\binom{n}{r-2}=O(1/n)$,
so in particular $\max_t\alpha_t<1/2$, say, if $n$ is large enough.
Let
\begin{equation}\label{bdef}
 \beta_t=\prod_{i=1}^{t}(1-\alpha_i).
\end{equation}
Then
\begin{equation}\label{bsmall}
 \beta_t = \exp(-O(t/n))
\end{equation}
uniformly in $0\le t\le n$. In particular, there is a constant $\beta>0$
such that for $n$ large enough,
\[
 \beta\le \beta_t\le 1
\]
for all $0\le t\le n$.
Set
\[
 x_t=x_{n,t}=n-t-n\beta_t.
\]
We showed in~\cite{BR_hyp} that this deterministic sequence is a good approximation
to the expected trajectory of the random process $(X_t)_{0\le t\le n}$,
and that $(x_t)$ is in turn well approximated by a certain (convex) continuous function.
We now give the details of these approximations.

Given an integer $r\ge 2$ and a positive real number $\la$, define
the function $g=g_{r,\la}$ on $[0,1]$ by
\begin{equation}\label{gdef}
 g(\tau)=g_{r,\la}(\tau) = 1-\tau-\exp\left(-\frac{\la}{r-1}(1-(1-\tau)^{r-1})\right).
\end{equation}
Since $\la$ depends on $n$, we have a different function $g_n$ for each $n$. As usual,
we suppress the dependence on $n$ in the notation.

\begin{lemma}\label{gprops}
Suppose that our Standard Assumption~\ref{A1}, holds.
Define a function $g=g_n$ as in \eqref{gdef}. Then
\begin{equation}\label{xg}
 x_t = n g(t/n)+O(1)
\end{equation}
uniformly in $0\le t\le n$.
Also,
\begin{equation}\label{g0}
 g(0)=0, \quad g'(0)=\la-1, \quad g''(\tau)\le 0, \hbox{\quad and\quad} \sup_{\tau\in [0,1]}|g''(\tau)|=O(1),
\end{equation}
and, writing $\rho$ for $\rho_{r,\la}$,
\begin{equation}\label{grho}
 g(\rho)=0\hbox{\quad and\quad}g'(\rho)=-(1-\las) = -\Theta(\eps).
\end{equation}
\end{lemma}
\begin{proof}
The proof is just elementary calculation. The calculations giving \eqref{xg}
and \eqref{g0} are described in~\cite{BR_hyp} (see equations (15) and (16) there),
so we omit them.
The final statement \eqref{grho} follows easily from
from \eqref{rkldef}, simple calculations and, for the final equality,
\eqref{lasrate2} (recalling that $\la$ is bounded by assumption).
\end{proof}

\begin{corollary}\label{cgvals}
Suppose that our Standard Assumption~\ref{A1} holds. Then 
there are constants $0<c_2,c_3<1$ (which may depend as usual on the choice
of the function $\eps(n)$, but not on $n$) such that for all $n$
and all $0\le \tau\le c_2\eps$ we have
\begin{equation}\label{gvals}
 g(\tau) \ge c_3\eps\tau,\hbox{\quad} g(\rho-\tau)\ge c_3\eps\tau,
\hbox{\quad and\quad} g(\rho+\tau) \le -c_3\eps\tau,
\end{equation}
where $g=g_n$ is defined in \eqref{gdef}.
\end{corollary}
\begin{proof}
Immediate from \eqref{g0} and~\eqref{grho}.
\end{proof}

We resume our analysis of the exploration process,
filtration $(\cF_t)$, and random sequences $(X_t)$ and $(\eta_t)$
introduced in Section~\ref{sec_expl}, next considering the martingale approximation to $(X_t)$.
Define $\beta_t=\beta_{n,t}$ as in \eqref{bdef}, and $\Delta_t=\eta_t-\E[\eta_t\mid\cF_{t-1}]$
as in \eqref{Dtdef}.
Set
\begin{equation}\label{Stdef}
 S_t= \sum_{i=1}^t\beta_i^{-1}\Delta_i\hbox{\quad and\quad} \tX_t=x_t+\beta_tS_t.
\end{equation}
Then $(S_t)$ is a martingale with respect to $(\cF_t)$, since $\beta_i$ is
deterministic and $\Delta_i$ is $\cF_i$-measurable with $\E[\Delta_i\mid \cF_{i-1}]=0$.
It follows that $(S_t)$ is a unlikely to be very large.

\begin{lemma}\label{Sconc}
Suppose that $r\ge 2$ is fixed and $p=p(n)=\Theta(n^{-r+1})$.
For any $1\le t=t(n)\le n$ and $y=y(n)=O(t)$ 
we have
\[
 \Pr\bb{ \max_{i\le t} |S_i| \ge y} \le 2\exp(-\Omega(y^2/t)).
\]
\end{lemma}
\begin{proof}
Note that
\[
 S_i-S_{i-1}=\beta_i^{-1}\Delta_i = \beta_i^{-1} \eta_i - \E[\beta_i^{-1}\eta_i\mid \cF_{i-1}],
\]
and that $\beta_i\ge\beta >0$. The result thus follows from Lemma~\ref{cdist},
Lemma~\ref{lbintail} (applied to the conditional distribution of $\beta_i^{-1}\eta_i$ given $\cF_{i-1}$),
and Lemma~\ref{lmart}.
\end{proof}

To close this section we quote Lemma 3 from~\cite{BR_hyp}.
This result shows that $\tX_t=x_t+\beta_tS_t$ is a very good approximation to $X_t$.
Recall from Section~\ref{sec_expl} that $C_t$ is the number of components that we have
started to explore by time $t$.

\begin{lemma}\label{lXXt}
Suppose that $r\ge 2$ is fixed and $p=p(n)=\Theta(n^{-r+1})$. Then there
there is a constant $c_1>0$ such that for all $n\ge 1$ we have
\begin{equation}\label{XXt}
 |X_t-\tX_t| \le c_1tC_t/n
\end{equation}
for $0\le t\le n$.\noproof
\end{lemma}

\section{Large deviations in the supercritical case}\label{sec_sup}

In this section we shall prove Theorem~\ref{thsupertail}. First, we give a definition
and two lemmas; these will be used in the next section also.
Throughout this section we assume our Standard Assumption~\ref{A1},
that $p=p(n)=\la(n) (r-2)! n^{-r+1}$, where $\la(n)=1+\eps(n)$
with $\eps>0$, $\eps=O(1)$ bounded, and $\eps^3n\to\infty$ as $n\to\infty$.
We explore the random hypergraph $\Hrnp$ as in Section~\ref{sec_expl},
and consider the filtration $(\cF_t)$ and random sequences $(X_t)$, $(A_t)$ and $(C_t)$
associated to this exploration. We shall also consider the deterministic
sequence $(x_t)$, function $g$, and martingale $(S_t)$ defined in Section~\ref{sec_mart}.

\begin{definition}\label{ZTT}
Given a deterministic `cut-off' $t_0=t_0(n)$, let
\begin{eqnarray*}
 Z &=& -\inf\{X_t:t\le t_0\}, \\
 T_0 &=& \inf\{t:X_t=-Z\} \hbox{\quad and} \\
 T_1 &=& \inf\{t:X_t=-Z-1\}.
\end{eqnarray*}
Thus $Z$ is the number of components completely explored by time $t_0$,
$T_0$ is the time at which we finish exploring the last such component, and $T_1$
is the time at which we finish exploring the next component.
Note that $Z+1=C_{t_0+1}$, and that by definition $T_0\le t_0<T_1$.
\end{definition}

We continue following the strategy of~\cite{BR_hyp}, itself based on that of~\cite{BR_walk}, modifying the calculations
to obtain the tighter error bounds claimed in Theorem~\ref{thsupertail}. The next lemma shows that we are unlikely
to see too many components near the start of the process.

\begin{lemma}\label{lZ}
Suppose that our Standard Assumption~\ref{A1} holds.
Let $t_0=t_0(n)$ satisfy $1\le t_0\le \min\{n/(2c_1), c_2\eps n\}$,
where $c_1$ is the constant in Lemma~\ref{lXXt} and $c_2$ is that in Corollary~\ref{cgvals}.
Then for any $y=y(n)$ satisfying $y\to\infty$ and $y=O(t_0)$ we have
\[
 \Pr\bb{ C_{t_0} \ge y } \le 2\exp(-\Omega(y^2/t_0)).
\]
\end{lemma}
\begin{proof}
Define $Z$ and $T_0$ as in Definition~\ref{ZTT}, and $(S_t)$ as in \eqref{Stdef}.
Let $\cA$ be the event
\[
 \cA =  \bigl\{ |S_t| \le y/4 \hbox{ \ for all \ }0\le t\le t_0 \bigr\}.
\]
By Lemma~\ref{Sconc} we have
$\Pr(\cA^\cc) \le 2\exp(-\Omega(y^2/t_0))$.

Since $t_0\le n/(2c_1)$, Lemma~\ref{lXXt} implies that for $t\le t_0$ we have $|X_t-\tX_t|\le C_t/2$.
Since $X_{T_0}=-Z$, $C_{T_0}=Z$ and $T_0\le t_0$, it follows that $\tX_{T_0}\le -Z/2$.
Since $T_0\le t_0\le c_2\eps n$, from \eqref{gvals} we have $g(T_0/n)\ge 0$.
By \eqref{xg} it follows that $x_{T_0}\ge -O(1)$,
so from \eqref{Stdef} we have $\beta_{T_0}S_{T_0}=\tX_{T_0}-x_{T_0} \le -Z/2+O(1)$.
Hence $S_{T_0}\le -Z/2+O(1)$.
By the definition of $\cA$, it follows that whenever $\cA$ holds, then
\[
 C_{t_0} \le Z+1 \le 2|S_{T_0}|+O(1) \le y/2 +O(1) < y
\]
for $n$ large enough.
\end{proof}

By our Standard Assumption~\ref{A1}, we have $\la=\la(n)=O(1)$ and $\la>1$.
Hence, by \eqref{lasrate2}, there is a constant $c_0$ such that
\begin{equation}\label{lasrate3}
  \las \le 1-c_0\eps.
\end{equation}
There exist a constant $c$ and an integer $n_0$ such that for all $n\ge n_0$ we have
\begin{equation}\label{cen}
 c\eps n\le  \min\{c_0\eps n/(4(r-1)\la),\ c_2\eps n,\ n/(2c_1),\ \rho_{r,\la}n/4 \bigr\},
\end{equation}
where $c_1$ is as in Lemma~\ref{lXXt} and $c_2$ as in Corollary~\ref{cgvals}.
Indeed, $\rho_{r,\la}=\Theta(\eps)$ from \eqref{rsasymp}, and $\la$ and $\eps$ are $O(1)$ by assumption,
so all terms on the right are $\Omega(\eps n)$. From now on, we shall always assume $n\ge n_0$.
In addition to the function $\eps(n)$, we fix a function $\omega(n)$ satisfying
\begin{equation}\label{ocs}
 \omega=\omega(n)\to\infty \hbox{\quad and\quad} \omega \le c\sqrt{\eps^3n}
\end{equation}
with $c$ as in \eqref{cen}.
Any new constants introduced may depend on the choice of $\omega(n)$ as well as that of $r$ and $\eps(n)$.

We shall work with the `initial cut-off' 
\begin{equation}\label{t0def}
  t_0 = \omega\sqrt{n/\eps},
\end{equation}
ignoring the rounding to integers, which causes no complications.
Since $n\ge n_0$, from \eqref{cen} we have
\begin{equation}\label{t0bds}
 t_0 \le  \min\{c_0\eps n/(4(r-1)\la),\ c_2\eps n,\ n/(2c_1),\ \rho_{r,\la}n/4 \bigr\}.
\end{equation}

Recalling \eqref{rsasymp}, set
\begin{equation}\label{t1def}
  t_1=\rho_{r,\la} n = \Theta(\eps n).
\end{equation}
Note for later that, from \eqref{grho}, $g(\rho_{r,\la})=0$, so \eqref{xg} implies that
\begin{equation}\label{xt1}
 x_{t_1} = O(1).
\end{equation}
The convex function $g(\tau)$ is positive on $(0,\rho)$ and passes through zero at $\tau=\rho$.
Hence, roughly speaking, we expect that near $t=t_1=\rho n$ the random trajectory
$(X_t)$ will be close to 0, and that around this point it will reach a new record
low value. We shall show that with high probability
this happens within $t_0$ steps of~$t_1$.

\begin{lemma}\label{lT1}
Suppose that our Standard Assumption~\ref{A1} holds and that $\omega(n)$ satisfies \eqref{ocs}.
Defining $t_0$ and $t_1$ as above and $T_1$ as in Definition~\ref{ZTT}, we have
\[
 \Pr\bb{ t_1-t_0\le T_1 \le t_1+t_0 } = 1-\exp(-\Omega(\omega^2)).
\]
\end{lemma}
\begin{proof}
As above, let $c_1$ and $c_3$ be the constants in Lemma~\ref{lXXt} and Corollary~\ref{cgvals},
and define $Z$, $T_0$ and $T_1$ as in Definition~\ref{ZTT}. 
Let $\cB_1$ be the event
\begin{equation*}
 \cB_1= \bigl\{ C_{t_0} \le c_3\eps t_0/(4\max\{c_1,1\}) \bigr\}.
\end{equation*}
By Lemma~\ref{lZ}, 
\[
 \Pr(\cB_1^\cc) \le 2\exp(-\Omega(\eps^2t_0))=2\exp(-\Omega(\omega\sqrt{\eps^3n})) = \exp(-\Omega(\omega^2)).
\]
Let $\cB_2$ be the event
\[
 \cB_2 =  \bigl\{ |S_t| \le c_3\eps t_0/5 \hbox{ \ for all \ }0\le t\le t_1+t_0 \bigr\}.
\]
Since $t_1=\Theta(\eps n)$ and $c_3\eps t_0/5=O(\eps^2 n)=O(\eps n)$, by Lemma~\ref{Sconc} we have
\[
 \Pr(\cB_2^\cc) \le 2\exp\Bb{-\Omega\bb{ \eps^2t_0^2/ (\eps n) } }
 =\exp(-\Omega(\omega^2)),
\]
since $t_0=\omega\sqrt{n/\eps}$.

To complete the proof of the lemma
we shall establish the deterministic claim that, for $n$ large enough,
\begin{equation}\label{claim}
 \cB_1\cap \cB_2 \implies t_1-t_0\le T_1\le t_1+t_0.
\end{equation}

To see this, suppose that $\cB_1$ and $\cB_2$ hold.
For $t\le \min\{T_1,t_1+t_0\}$ relations \eqref{xg} and \eqref{XXt} and
the definition \eqref{Stdef} give
\begin{eqnarray*} 
 |ng(t/n)-X_t| &\le& |ng(t/n)-x_t|+|\tX_t-x_t| +|X_t-\tX_t| \\
 &\le& O(1) + \beta_t|S_t| +c_1C_{t_0}  \\
 &\le& O(1) + c_3\eps t_0/5 + c_3\eps t_0/4,
\end{eqnarray*}
using $\beta_t\le 1$ and the assumption that $\cB_1\cap\cB_2$ holds in the last step.
Hence
\begin{equation}\label{xX}
 |ng(t/n)-X_t| \le c_3\eps t_0/2
\end{equation}
for $n$ large enough.

Suppose for a contradiction that $T_1<t_1-t_0$.
Then \eqref{xX} applies for $t=T_1\in [t_0,t_1-t_0]$.
From \eqref{g0} the function $g$ is concave. Hence, from \eqref{gvals},
we have $ng(t/n)\ge c_3\eps t_0$ for $t\in [t_0,t_1-t_0]$, so
\[
0> -(Z+1) = X_{T_1} \ge c_3\eps t_0 - c_3\eps t_0/2 >0,
\]
a contradiction. Thus $T_1\ge t_1-t_0$.

Suppose instead that $T_1>t_1+t_0$. Then \eqref{xX} applies with $t=t_1+t_0$.
Hence, by the definition of $T_1=\inf\{t:X_t=-Z-1\}$, the last
bound in \eqref{gvals} and \eqref{xX},
\[
 -Z \le X_{t_1+t_0} \le -c_3\eps t_0+c_3\eps t_0/2 = -c_3\eps t_0/2.
\]
Thus $C_{t_0}\ge Z\ge c_3\eps t_0/2$, contradicting the assumption that $\cB_1$ holds.
This completes the proof of \eqref{claim} and thus of the lemma.
\end{proof}

We are now ready to prove Theorem~\ref{thsupertail}.
\begin{proof}[Proof of Theorem~\ref{thsupertail}]
The conditions of Theorem~\ref{thsupertail} include our Standard Assumption~\ref{A1}, which we thus assume.
The conditions also state that $\omega=\omega(n)$ satisfies
$\omega\to\infty$ and $\omega=O(\sqrt{\eps^3 n})$.
To apply the lemmas above we  need the additional condition \eqref{ocs}, i.e.,
$\omega\le c\sqrt{\eps^3n}$ with $c$ as in \eqref{cen}.
We may impose this without problems since, in proving \eqref{st1}, we may reduce $\omega$ by a constant
factor, changing the implicit constant to compensate.
As in \eqref{t0def} and \eqref{t1def}, we set $t_0=\omega\sqrt{n/\eps}$ and $t_1=\rho n$.
Our first aim is to show that
\[
  \Pr\Bb{ |L_1(\Hrnp)-t_1| > 2t_0 } = \exp(-\Omega(\omega^2)),
\]
which (changing $\omega$ by an irrelevant factor of $3$, say) is exactly \eqref{st1}.

Let $\cC$ be the component that we explore from time $T_0+1$ to time $T_1$.
We have $T_0\le t_0$ by definition,
while from Lemma~\ref{lT1}, with probability $1-\exp(-\Omega(\omega^2))$ we have $|T_1-t_1|\le t_0$.
Thus $\cC$ has between $t_1-2t_0$ and $t_1+t_0$ vertices.
Moreover, since $t_0\le t_1/4$ by \eqref{t0bds}, any component explored before $\cC$ has at most $t_0<|\cC|$ vertices. To complete the proof of \eqref{st1} it remains to show that with very 
high probability no component explored after time $T_1$ has more than $|\cC|$ vertices.

Stopping the exploration at time $T_1$, the unexplored part of $\Hrnp$ has
exactly the distribution of $H^r_{n-T_1,p}$.
We shall apply Theorem~\ref{thsubtail} to this hypergraph; to obtain the result we
need we must show that its `branching factor'
\[
 \La = (n-T_1)^{r-1}p/(r-2)! = \la (1-T_1/n)^{r-1} 
\]
is $1-\Omega(\eps)$. Since $T_1\ge t_1-2t_0=\rho n-2t_0$, we have
\begin{multline*}
 \La  \le \la (1-\rho+2t_0/n)^{r-1} \\
 \le \la (1-\rho)^{r-1} + 2(r-1)\la t_0/n 
 \le \la (1-\rho)^{r-1} + c_0\eps/2,
\end{multline*}
using the first condition in \eqref{t0bds} in the last step.
By \eqref{las} and \eqref{rkldef} we have 
\[
 \la (1-\rho)^{r-1} = \la (1-\rho_{r,\la})^{r-1} = \la (1-\rho_\la) = \las,
\]
so, recalling \eqref{lasrate3},
\[
 \La\le 1-c_0\eps+c_0\eps/2 = 1-\Omega(\eps).
\]
Hence, by Theorem~\ref{thsubtail} (applied with $n-T_1=\Theta(n)$ in place of $n$
and $1-\La=\Omega(\eps)$ in place of $\eps$, and with $L=|\cC|=\Theta(\eps n)$), with probability
$1-\exp(-\Omega(\eps^3n))=1-\exp(-\Omega(\omega^2))$, the hypergraph $H^r_{n-T_1,p}$
has no component with at least as many vertices as $\cC$. It follows that
\begin{equation}\label{big}
 \Pr\bb{ \cC\hbox{ is the unique largest component of }\Hrnp } = 1-\exp(-\Omega(\omega^2)),
\end{equation}
completing the proof of \eqref{st1}.

The bound \eqref{st2} follows easily from \eqref{st1}, Theorem~\ref{thsubtail} and a standard duality argument;
let us outline this briefly. Condition not only on the number $L_1$ of vertices
in the largest component $\cL_1$ of $H=\Hrnp$, but also on the vertex set of this component.
The conditional distribution of $H^-=H-\cL_1$ is then that of $H^r_{n-L_1,p}$ conditioned
on a monotone decreasing event (that there is no component with more than $L_1$ vertices, plus an
extra condition to deal with the possibility of ties; see, e.g.,~\cite[Section 8]{BRsmoothing}).
Taking $\omega=c\sqrt{\eps^3n}$ with $c$ as in \eqref{cen}, so $t_0=c\eps n$,
as above we have $|L_1-t_1|\le 2t_0$ with probability $1-\exp(-\Omega(\omega^2))=1-\exp(-\Omega(\eps^3n))$.
It follows as above that the `branching factor' of $H^r_{n-L_1,p}$ is $1-\Omega(\eps)$ 
(in fact $1-\Theta(\eps)$, but we only need an upper bound).
Since conditioning on a decreasing event can only decrease the probability of having
a component of more than a given size, we may apply Theorem~\ref{thsubtail} to see
that $\Pr(L_1(H^-)\ge L)\le C\eps n L^{-1}\exp(-\eps^2L/C)$. By assumption $L=O(\eps n)$,
so increasing $C$ if necessary we may absorb the additional $\exp(-\Omega(\eps^3n))$
error probability into the expression in \eqref{st2}.
\end{proof}

\section{Bivariate central limit theorem}\label{sec_biv}

\subsection{Martingale CLTs}

In this section we shall prove Theorem~\ref{thglobal2}. For this we need a martingale central
limit theorem. Although the result we need is well known, there are many possible variants,
and it is not so easy to find a form convenient
for combinatorial applications in the literature; the following is (up to a trivial change noted below)
Corollary 1 of Brown and Eagleson~\cite{BrownEagleson}; we thank Svante Janson for supplying
this reference.

\begin{lemma}\label{lm1}
For each $n$, let $(M_{n,t})_{t=0}^{k(n)}$ be a martingale with respect to a filtration
$(\cF_{n,t})$, with $M_{n,0}=0$ for all $n$. Writing $\Delta_{n,t}=M_{n,t}-M_{n,t-1}$, let
\[
 V_n = \sum_{t=1}^{k(n)} \Var[ \Delta_{n,t} \mid \cF_{t-1}]
\]
be the sum of the conditional variances of the increments.
Suppose that
\begin{equation}\label{VC}
 V_n\pto \sigma^2 
\end{equation}
as $n\to\infty$, where $\pto$ denotes convergence in probability. Suppose also
that for any constant $\delta>0$ we have
\begin{equation}\label{LC}
 \sum_{t=1}^{k(n)} \E[ \Delta_{n,t}^2 \ind{|\Delta_{n,t}|\ge \delta} \mid \cF_{t-1}]  \pto 0.
\end{equation}
Then $M_{n,k(n)}\dto N(0,\sigma^2)$.
\end{lemma}
The only difference between the statement above and Corollary 1 in~\cite{BrownEagleson}
is that there $k(n)=n$, i.e., the array is triangular. As noted in~\cite{BrownEagleson},
this loses no generality, since $n$ plays no role in the result above except as an index.
(Thus we may pad rows with zeros and/or add zero rows to transform a general array into
a triangular one.)
Condition \eqref{LC} is the `Lindeberg' condition, in a conditional form.

Lemma~\ref{lm1} extends without problems to higher dimensions, i.e., to simultaneous
convergence of several martingales; we shall need the following two-dimensional version.
\begin{lemma}\label{lm2}
For each $n\ge 1$ and $j\in \{1,2\}$ let $(M_{j,n,t})_{t=0}^{k(n)}$
be a martingale with respect to a filtration
$(\cF_{n,t})$, with $M_{j,n,0}=0$. Writing $\Delta_{j,n,t}=M_{j,n,t}-M_{j,n,t-1}$,
suppose that the Lindeberg condition~\eqref{LC} holds for $j=1$ and for $j=2$, and that
\begin{equation}\label{VC2}
 V_{j,n} = \sum_{t=1}^{k(n)} \Var[ \Delta_{j,n,t} \mid \cF_{t-1}] \pto \sigma_j^2
\end{equation}
for $j=1,2$ and
\begin{equation}\label{CVC2}
 V_{1,2,n} = \sum_{t=1}^{k(n)} \Covar[ \Delta_{1,n,t},\Delta_{2,n,t} \mid \cF_{t-1}] \pto \sigma_{1,2}.
\end{equation}
Then $(M_{1,n,k(n)},M_{2,n,k(n)})$ converges in distribution to a bivariate normal distribution
$(N_1,N_2)$ with $N_j\sim N(0,\sigma_j^2)$ and $\Covar[N_1,N_2]=\sigma_{1,2}$.
\end{lemma}
\begin{proof}
By the Cram\'er--Wold Theorem~\cite{CramerWold} (see e.g., Billingsley~\cite[Theorem 29.4]{Billingsley}),
a sequence of random vectors converges in distribution to a given random vector if and only if
all the one-dimensional projections converge in distribution.
Thus it suffices to show that for any constants $\alpha$ and $\beta$,
$\alpha M_{1,n,k(n)} + \beta M_{2,n,k(n)}$ converges in distribution
to a Gaussian with mean 0 and the appropriate variance, namely $\alpha^2\sigma_1^2+2\alpha\beta \sigma_{1,2}+\beta^2\sigma_2^2$. This follows by applying Lemma~\ref{lm1} to $M_{n,t} = \alpha M_{1,n,t}+\beta M_{2,n,t}$.
Indeed, using the formula $\Var[\alpha X+\beta Y]=\alpha^2 \Var[X]+2\alpha\beta \Covar[X,Y]+\beta^2\Var[Y]$,
which applies just as well to conditional variances, the variance condition \eqref{VC}
follows from the assumptions on $V_{1,n}$, $V_{2,n}$ and $V_{1,2,n}$.
In establishing the Lindeberg condition we may assume without loss of generality that $\alpha=\beta=1$.
It is easy to see that
\[
 (X+Y)^2 \ind{|X+Y|\ge 2\delta} \le 4X^2 \ind{|X|\ge \delta} + 4Y^2\ind{|Y|\ge \delta}.
\]
(Indeed, the first or second term on its own is an upper bound according to whether $|X|\ge |Y|$
or $|X|<|Y|$.)
Hence the Lindeberg condition for $(M_{1,n,t}+M_{2,n,t})$ follows from
the same condition for $(M_{1,n,t})$ and $(M_{2,n,t})$.
\end{proof}

\subsection{Application to $\Hrnp$}

Let $N_t$ denote the nullity of the hypergraph formed by all edges exposed
within the first $t$ steps of the exploration described in Section~\ref{sec_expl}.
Since nullity is additive over components, the component $\cC$ explored
between time $T_0$ and $T_1$ has nullity $N_{T_1}-N_{T_0}$. We now study
the joint distribution of this quantity and $T_1-T_0$.

In this section we assume the following stronger form of our Standard Assumption~\ref{A1}.
\begin{assumption2}\label{A2}
The integer $r\ge 2$ is fixed, $\eps=\eps(n)$ is a function satisfying $\eps>0$, $\eps\to 0$
and $\eps^3 n\to\infty$. Furthermore, $\la=\la(n)=1+\eps$ and $p=p(n)=\la (r-2)!n^{-r+1}$.
\end{assumption2}

As usual, we consider the exploration process defined in Section~\ref{sec_expl}, and
the associated random sequences $(X_t)$, $(A_t)$ and $(\eta_t)$.
Recall that $\cA_t$
denotes the set of active vertices at time $t$, and $A_t=|\cA_t|$.
Let $E_t$ be the set of edges revealed during step $t$. Then,
whether or not we start exploring a new component in step $t$, we have
\[
 N_t-N_{t-1} = (r-1)|E_t| - \left|\bigcup_{e\in E_t} (e\setminus\{v_t\})\setminus \cA_{t-1} \right|.
\]
Indeed, we have added $|E_t|$ edges to the `revealed graph', and the vertices in the union above,
which were previously isolated, have now been connected to $v_t$. (If $\cA_{t-1}\ne\emptyset$,
then the vertices in $\cA_{t-1}$ were already in the same component as $v_t$.)

Let $\xi_t$ be the number of vertices in $\cA_{t-1}\setminus\{v_t\}$ included in one or more
edges in $E_t$, and set
\[
  \zeta_t = \sum_{e,f\in E_t} |(e\cap f)\setminus\{v_t\}|,
\]
where the sum is over all unordered pairs of distinct edges in $E_t$. Then
\begin{equation}\label{xz}
 \xi_t\le N_t-N_{t-1} \le \xi_t + \zeta_t.
\end{equation}
Considering the number of triples $(e,w,f)$ where $e$ and $f$ are edges tested
at step $t$ and $w\in (e\cap f)\setminus\{v_t\}$, by linearity of expectation we have
\begin{equation}\label{eZ}
 \E[\zeta_t] \le \binom{n-t}{r-1}(r-1)\binom{n-t-1}{r-2} p^2 = O(n^{-1}).
\end{equation}
As we shall see later, this implies that we can essentially ignore $\zeta_t$, and consider only the $\xi_t$.

Let $A_t'=|\cA_t\setminus\{v_{t+1}\}|$ be the number of active vertices after $t$ steps
other than $v_{t+1}$.
Thus $A_t'=A_t-1$ if $A_t\ne 0$ and $A_t'=0$ if $A_t=0$. In particular, $A_t'=A_t+O(1)$.
Let $\pi_t=\pi_{1,t}$ be the probability that a given vertex $u$ 
not among $v_1,\ldots,v_t$ is contained in $\bigcup_{e\in E_t}e$.
(This quantity is denoted $\pi_1$ in \cite{BR_hyp}.)
Since there are $c_t=\binom{n-t-1}{r-2}$
edges tested at step $t$ that contain $u$, we have
\begin{multline}\label{pi}
 \pi_t=1-(1-p)^{c_t} = pc_t+O(p^2c_t^2) = pc_t +O(1/n^2) \\
 = \la (1-t/n)^{r-2}/n+O(1/n^2),
\end{multline}
recalling that $p=\la (r-2)!n^{-r+1}$, with $\la=\la(n)=1+\eps$.
In particular, for $t=O(\eps n)$ we have
\begin{equation}\label{pismall}
 \pi_t = (1+O(\eps))/n.
\end{equation}
From the definition of $\xi_{t+1}$ and the linearity of expectation,
\begin{equation}\label{Exi}
 \E[\xi_{t+1}\mid \cF_t] = A_t'\pi_{t+1} = A_t\pi_t+O(1/n).
\end{equation}

Let $\pi_{2,t}$ be the probability that two given (distinct) vertices $u,w\in[n]\setminus\{v_1,\ldots,v_t\}$
are contained in $\bigcup_{e\in E_t}e$. Considering the cases where
$u$, $w$ are in the same $e\in E_t$ and in distinct $e,f\in E_t$ it is easy to see that
\[
 \pi_{2,t}\le p\binom{n-t-2}{r-3} + \pi_{1,t}^2 = O(1/n^2).
\]
It follows that
\begin{equation}\label{Exixi}
 \E[\xi_{t+1}(\xi_{t+1}-1) \mid \cF_t ] = A_t'(A_t'-1)\pi_{2,t+1} = O((A_t/n)^2).
\end{equation}
Similarly,
\begin{equation}\label{Exieta}
 \E[\xi_{t+1}\eta_{t+1} \mid \cF_t ] = A_t'(n-t-1-A_t')\pi_{2,t+1} \le nA_t'\pi_{2,t+1} = O(A_t/n).
\end{equation}

These bounds are enough to extend the argument we used in~\cite{BR_hyp} to prove
a univariate central limit theorem for $L_1(\Hrnp)$,
to prove Theorem~\ref{thglobal2}. Roughly speaking, we shall use
the estimates above to decompose $(N_t)$ into two parts. The first part is a martingale
that is essentially independent of $(X_t)$, and the second depends on $(X_t)$ in a simple
way. Then we can apply Lemma~\ref{lm2} to prove the result.
As usual in this type of argument, we must calculate the expectation terms very accurately,
but it suffices to estimate the variance terms within a factor of $1+o(1)$.

For the rest of the paper we consider $\eps=\eps(n)$ satisfying 
our Strong Assumption~\ref{A2},
and a function $\omega=\omega(n)$ satisfying
\[
 \omega\to\infty\hbox{\quad with\quad}\omega=o((\eps^3n)^{1/6}).
\]
Define $t_1=\rho_{r,\la} n$ as before (in \eqref{t1def}), recalling that $t_1=O(\eps n)$.
As before, set
\[
 t_0=\omega\sqrt{n/\eps}.
\]
In addition, define $(S_t)$ as in \eqref{Stdef}, and $Z$, $T_0$ and $T_1$ as in Definition~\ref{ZTT}.
We shall work with these quantities for the rest of the paper.

As usual, we say that an event $\cE=\cE(n)$ holds \emph{whp} (\emph{with high probability}),
if $\Pr(\cE(n))\to 1$ as $n\to\infty$.

\begin{lemma}\label{props}
Let
\begin{eqnarray*}
 \cE_1 &=& \{ Z \le \omega^{-1}\sqrt{\eps n} \hbox{ \ and \ } T_0 \le \omega^{-1}\sqrt{n/\eps} \}, \\
 \cE_2 &=& \bigl\{ \max_{t\le t_1+t_0} |S_t| \le \omega\sqrt{\eps n} \bigr\}, \hbox{ and} \\
 \cE_3 &=& \{ t_1-t_0\le T_1 \le t_1+t_0 \},
\end{eqnarray*}
and set $\cE=\cE_1\cap\cE_2\cap \cE_3$. Then $\cE$ holds whp.
\end{lemma}
\begin{proof}
Under our Standard Assumption~\ref{A1}, which of course is implied by our Strong Assumption~\ref{A2}, we proved
in \cite{BR_hyp} that $\cE_1$ holds whp -- see the paragraph after (20) on page 448 of~\cite{BR_hyp}.

For $\cE_2$ apply Lemma~\ref{Sconc}, noting that $t_1+t_0=\Theta(\eps n)$,
and that $\omega\sqrt{\eps n}=O(\eps n)$, since $\omega=o((\eps^3 n)^{1/6}) = O(\sqrt{\eps^3 n})= O(\sqrt{\eps n})$,
with room to spare.

Finally, $\cE_3$ holds whp by Lemma~\ref{lT1}.
\end{proof}

For the rest of the paper
the events $\cE_i$ and $\cE$ are as above.
In our next lemma we establish some consequences of the event $\cE$ holding.
Let
\[ 
 I=[t_1-t_0,t_1+t_0].
\]
\begin{lemma}\label{Asmall}
If $\cE$ holds then

(i) $C_{t_1+t_0}=O(\omega\sqrt{\eps n})$,

(ii) $
 \max_{t\in I} A_t,\ \max_{t\le t_0} A_t =O(\omega\sqrt{\eps n}),
$ and

(iii) $\max_{t\le t_0+t_1} A_t = O(\eps^2 n)$.
\end{lemma}
\begin{proof}
Suppose $\cE=\cE_1\cap \cE_2\cap \cE_3$ holds.
Since $C_{T_1}=Z+1$ and by assumption $T_1\in I$ we have, very crudely,
that
\[
 C_{t_1+t_0}\le C_{T_1}+(t_1+t_0-T_1) \le Z+1+2t_0 \le 3\omega\sqrt{n/\eps}.
\]
Since $t_1+t_0=O(\eps n)$, it follows from Lemma~\ref{lXXt} that $|X_t-\tX_t| = O(\omega \sqrt{\eps n})$,
uniformly in $t\le t_1+t_0$.
Recalling \eqref{g0} and \eqref{grho},
for $t\le t_0$ or $t\in I$ we have
$g(t/n)=O(\eps t_0/n)$ and hence
$x_t=n g(t/n)+O(1)=O(\eps t_0)=O(\omega\sqrt{\eps n})$.
Since $\cE_2$ holds it follows that
\begin{equation}\label{mX}
 |X_t|\le |x_t|+|S_t|+|X_t-\tX_t|=O(\omega\sqrt{\eps n}),
\end{equation}
uniformly in $t\in [0,t_0]\cup I$. Let $T=\max\{t\in I:A_t=0\}$ be the last time that
we finish exploring a component within the interval $t\in I$;
this makes sense since $T_1\in I$. Then $A_T=0$ so $C_T=-X_T$. Hence
$C_{t_1+t_0}\le C_T+1=O(\omega\sqrt{\eps n})$, proving (i).

For $t\le t_1+t_0$ we have $A_t\le |X_t|+C_t\le |X_t|+C_{t_1+t_0}$.
Hence (ii) follows from (i) and \eqref{mX}.
Recalling from \eqref{rsasymp} that $\rho=\Theta(\eps)$, from \eqref{g0}
it is easy to check that $\sup_{\tau\le\rho}g(\tau)=O(\eps^2)$.
The argument for (iii) is very similar to that for (ii), using this estimate to show
that $x_t=O(\eps^2n)$ for $t\le t_1+t_0$, in place of the tighter
bound $O(\omega\sqrt{\eps n})$ we used in case (ii).
\end{proof}

In the rest of the paper we use the following standard notation for probabilistic asymptotics:
given random variables $(Z_n)$ and a function $f(n)>0$, we write $Z_n=\op(f(n))$ if $Z_n/f(n)$
converges to $0$ in probability as $n\to\infty$. We (briefly) write $Z_n=\Op(1)$ to mean that $Z_n$
is bounded in probability.

\begin{lemma}\label{Napp}
Let $\cC$ be the component explored between times $T_0$ and $T_1$. Then
\[
 n(\cC) =  N_{T_1}-N_{T_0} = \sum_{t=1}^{t_1} \xi_t + \op(\sqrt{\eps^3 n}).
\]
\end{lemma}
\begin{proof}
That $n(\cC)=N_{T_1}-N_{T_0}$ is immediate from
the additivity of nullity over components.
From \eqref{xz} we have
\[
 \left| N_{T_1}-N_{T_0} -\sum_{t=T_0+1}^{T_1} \xi_t \right| \le \sum_{t=T_0+1}^{T_1} \zeta_t
\le \sum_{t=1}^n \zeta_t= \Op(1)= \op(\sqrt{\eps n}),
\]
where for the second-last step we used the expectation bound \eqref{eZ} and Markov's inequality.

Since $\xi_t\ge 0$ and $T_0\le t_0$ hold by definition, whenever $\cE_3$ holds we have
\[
 \left| \sum_{t=T_0+1}^{T_1} \xi_t - \sum_{t=1}^{t_1} \xi_t \right| \le \sum_{t=1}^{t_0} \xi_t + \sum_{t=t_1-t_0+1}^{t_1+t_0} \xi_t = B,
\]
say. By Lemma~\ref{Asmall}(ii) and \eqref{Exi},
since $\max_t\pi_t=O(1/n)$, we have
\[
 \E[ 1_{\cE}B ] \le 3t_0 O(\omega\sqrt{\eps n})/n = O( \omega^2 ) = o(\sqrt{\eps^3 n}).
\]
Since $\cE$ holds whp, it follows that $B=\op(\sqrt{\eps^3 n})$.
\end{proof}

Following (a modified form of) the strategy in~\cite{norm2}, we now consider
the Doob decomposition of the sequence $(\sum_{i=1}^t\xi_i)$. More precisely,
writing (as before) $\cF_t$ for the $\sigma$-algebra generated by all information revealed up to step $t$
of the exploration process, set
\begin{equation}\label{Dtsdef}
 D_t^* = \E[\xi_t\mid \cF_{t-1}] \hbox{\quad and\quad} \Delta_t^* = \xi_t-D_t^*.
\end{equation}
\begin{lemma}\label{D*}
Define $\rho^*=\rho_{r,\la}^*$ as in \eqref{rhosdef}. Then
\[
 \sum_{t=1}^{t_1} D_t^* = \rho^*n + \sum_{t=1}^{t_1}\gamma_t\Delta_t + \op(\sqrt{\eps^3n}),
\]
where the $\gamma_t$ are deterministic and satisfy
\begin{equation}\label{gt}
 \gamma_t = \frac{t_1-t}{n} +O(\eps^2),
\end{equation}
uniformly in $1\le t\le t_1$.
\end{lemma}
\begin{proof}
From Lemma~\ref{lXXt} and the definitions $X_t=A_t-C_t$ and $\tX_t=x_t+\beta_tS_t$,
we have
\[
 A_t = X_t+C_t = \tX_t+(X_t-\tX_t)+C_t=  \tX_t + O(C_t) = x_t+\beta_t S_t+O(C_t).
\]
Hence, recalling \eqref{Exi},
\[
 D_{t+1}^* = \E[\xi_{t+1}\mid\cF_t] = x_t\pi_t + \beta_t S_t\pi_t +O(C_t/n)+O(1/n),
\]
so
\begin{equation}\label{Dtsum}
 \sum_{t=1}^{t_1} D_t^* = \sum_{t=0}^{t_1-1} x_t\pi_t 
 + \sum_{t=0}^{t_1-1} \beta_t S_t\pi_t + O(E),
\end{equation}
where $E=\sum_{t=0}^{t_1-1}C_t/n$. We shall estimate the terms on the right-hand
side of \eqref{Dtsum} in reverse order.

Whenever the event $\cE=\cE_1\cap\cE_2\cap\cE_3$ defined in Lemma~\ref{props} holds,
for $t\le t_1-t_0$ we have $C_t\le Z+1=O(\omega^{-1}\sqrt{\eps n})$. Also,
by Lemma~\ref{Asmall}(i), for $t_1-t_0<t\le t_1$ we have
$C_t\le C_{t_0+t_1}=O(\omega\sqrt{\eps n})$. Since $\cE$ holds whp, it follows
that whp
\[
 E = O\left((t_1-t_0)\omega^{-1}\sqrt{\eps n}/n + t_0 \omega\sqrt{\eps n}/n\right)
 = O\left(\omega^{-1}\sqrt{\eps^3n} + \omega^2\right) = o\left(\sqrt{\eps^3n}\right).
\]
Thus $E=\op(\sqrt{\eps^3n})$.

Turning to the middle term in the right-hand side of \eqref{Dtsum},
let
\[
  \gamma_i = \sum_{t=i}^{t_1-1} \frac{\beta_t\pi_t}{\beta_i}.
\]
From the definition \eqref{Stdef} of $S_t$, we have
\[
 \sum_{t=0}^{t_1-1} \beta_t S_t\pi_t
  = \sum_{t=1}^{t_1-1}\sum_{i=1}^t \frac{\beta_t\pi_t}{\beta_i}\Delta_i 
 = \sum_{i=1}^{t_1-1}\gamma_i \Delta_i.
\]
Recalling \eqref{bsmall} and \eqref{pismall},
\[
 \gamma_i
 = \sum_{t=i}^{t_1-1} \frac{1+O(\eps)}{n} = \frac{t_1-i}{n}+O(\eps^2).
\]

Finally, turning to the main term in \eqref{Dtsum},
we shall make use of the function $g=g(\tau)$ defined in \eqref{gdef},
and the related function
\[
 h(\tau) = g(\tau) \lambda (1-\tau)^{r-2}.
\]
From the definition \eqref{gdef} of $g$ and relations \eqref{xg} and \eqref{pi} we have
\begin{equation}\label{xp}
 x_t\pi_t = h(t/n)+O(1/n).
\end{equation}
An elementary calculation shows that
\begin{eqnarray*}
 \int_0^\rho h(\tau)\dd\tau &=&  \left[ \exp\left(-\frac{\la}{r-1}(1-(1-\tau)^{r-1})\right) - \frac{\la(1-\tau)^r}{r} \right]_{\tau=0}^\rho \\
 &=& \left[ 1-\tau-g(\tau) - \frac{\la(1-\tau)^r}{r} \right]_{\tau=0}^\rho, \\
\end{eqnarray*}
substituting in the definition \eqref{gdef} of $g$ for the second step. Recalling from \eqref{g0}
and \eqref{grho} that $g(\rho)=0=g(0)$, it follows that
\[
 \int_0^\rho h(\tau)\dd\tau = -\rho +\frac{\la}{r}(1-(1-\rho)^r) = \rho^*,
\]
where $\rho^*=\rho_{r,\la}^*$ is defined in \eqref{rhosdef}.
It is easy to check that $h'$ is uniformly bounded on $[0,1]$; 
it thus follows easily from \eqref{xp} that
\[
 \sum_{t=0}^{t_1-1}x_t\pi_t = n\int_0^\rho h(\tau)\dd\tau + O(t_1/n) = \rho^*n+o(1).
\]
Combining the estimates just proved, Lemma~\ref{D*} follows from \eqref{Dtsum}.
\end{proof}
We note the following simple corollary for later.
\begin{corollary}\label{Dscor}
We have
\[ 
 \sum_{t=1}^{t_1}D_t^* = (1+\op(1))\rho^*n.
\]
\end{corollary}
\begin{proof}
Recall that $\Delta_t$ is $\cF_t$-measurable with $\E[\Delta_t\mid\cF_{t-1}]=0$.
Hence,
\[
 \Var\left[ \sum_{t=1}^{t_1}\gamma_t\Delta_t \right] = \sum_{t=1}^{t_1} \gamma_t^2\Var[\Delta_t] 
= O(\eps^3n),
\]
since there are $t_1=O(\eps n)$ terms, each $\gamma_t$ is $O(\eps)$ from \eqref{gt},
and, from Lemma~\ref{cvar}, $\Var[\Delta_t]=O(1)$.
The result thus follows from Lemma~\ref{D*} and the observation that $\sqrt{\eps^3n}=o(\rho^*n)$,
recalling \eqref{rsasymp} and that $\eps^3n\to\infty$.
\end{proof}

After this preparation, we are ready to complete the proof of Theorem~\ref{thglobal2}.
\begin{proof}[Proof of Theorem~\ref{thglobal2}]
Suppose that our Strong Assumption~\ref{A2} holds, and that $\omega(n)$ satisfies \eqref{ocs}.
Define $t_0$ and $t_1$ as in \eqref{t0def} and \eqref{t1def}, and $Z$, $T_0$ and $T_1$
as in Definition~\ref{ZTT}.

Let $\cC$ be the component of $\Hrnp$ explored between times $T_0$ and $T_1$.
By \eqref{big}, whp $\cC$ is the unique component
$\cL_1$ of $\Hrnp$ with the most vertices. We need one final result
from \cite{BR_hyp}, namely Eq. (21) there, which says that
\begin{equation}\label{T1tX}
 T_1 = t_1 + \tX_{t_1}/(1-\las) + \op(\sqrt{n/\eps}).
\end{equation}
(The quantity $\sigma_0$ appearing in~\cite{BR_hyp} is simply $\sqrt{\eps n}$.)
Now $t_1=\rho n$ by definition. From Lemma~\ref{props} (considering $\cE_1$)
we have $T_0\le \omega^{-1}\sqrt{n/\eps}$ whp, and thus $T_0=\op(\sqrt{n/\eps})$.
Hence
\begin{align*}
 |\cC| = T_1-T_0 &= \rho n + \tX_{t_1}/(1-\las) + \op(\sqrt{n/\eps}) \\
 &= \rho n + \beta_{t_1}S_{t_1}/(1-\las) + \op(\sqrt{n/\eps}),
\end{align*}
since $\tX_{t_1}=x_{t_1}+\beta_{t_1}S_{t_1}=\beta_{t_1}S_{t_1}+O(1)$
by \eqref{Stdef} and \eqref{xt1}.
From \eqref{rsasymp}, we have $1-\las\sim \eps$, while from \eqref{bsmall}
we have $\beta_{t_1}\sim 1$. 
Recalling Lemma~\ref{Napp}, to complete the proof of Theorem~\ref{thglobal2}
it thus suffices to show that the pair
\[
 \left( S_{t_1} , \sum_{t=1}^{t_1} \xi_t -\rho^* n\right)
\]
is asymptotically bivariate normal with zero mean, variance $2\eps n$ for the first
coordinate, $\frac{10}{3(r-1)^2}\eps^3n$ for the second, and covariance 
$\frac{2}{r-1}\eps^2n$.

From \eqref{Dtsdef} and Lemma~\ref{D*},
\begin{equation}\label{xitsum}
 \sum_{t=1}^{t_1} \xi_t = \sum_{t=1}^{t_1} D^*_t + \sum_{t=1}^{t_1} \Delta^*_t
 = \rho^*n + \sum_{t=1}^{t_1} (\gamma_t\Delta_t + \Delta^*_t) + \op(\sqrt{\eps^3n}),
\end{equation}
where the $\gamma_t$ are deterministic and satisfy \eqref{gt}.
Set
\begin{equation}\label{hDt}
 \hat\Delta_t = \gamma_t\Delta_t+\Delta_t^* \hbox{\quad and\quad} \hat S_i=\sum_{t=1}^i\hat\Delta_t.
\end{equation}
Then \eqref{xitsum} implies that
\[
 \hat S_{t_1}
 = \sum_{t=1}^{t_1} \xi_t -\rho^* n +\op(\sqrt{\eps^3n}).
\]
Thus to prove Theorem~\ref{thglobal2} it suffices to show that $(S_{t_1},\hat S_{t_1})$ is asymptotically bivariate
normal with mean zero and variance as above. More precisely, it suffices to show that
\begin{equation}\label{concrete}
 \bb{ (\eps n)^{-1/2} S_{t_1}, (\eps^3 n)^{-1/2} \hat S_{t_1} } \dto (X,Y)
\end{equation}
where $(X,Y)$ is bivariate normal with
\begin{equation}\label{Vars}
 \Var[X]= \sigma_1^2 = 2,\quad \Var[Y]=\sigma_2^2 = \frac{10}{3(r-1)^2} \hbox{\quad and\quad} 
 \Covar[X,Y] = \sigma_{1,2} = \frac{2}{r-1}.
\end{equation}
For this we shall use Lemma~\ref{lm2}.

First, by the definitions \eqref{Dtdef} and \eqref{Dtsdef}, $\E[\Delta_t\mid \cF_{t-1}]=
\E[\Delta_t^*\mid \cF_{t-1}]=0$, so $\E[\hat\Delta_t\mid \cF_{t-1}]=0$,
and $(S_t,\hat S_t)_{t=0}^{t_1}$ is a martingale. The remaining assumptions of Lemma~\ref{lm2}
are captured in the following claim.
\begin{claim}\label{cc}
As $n\to\infty$ we have
\begin{eqnarray}
 \sum_{t=1}^{t_1} \Var[\beta_t^{-1}\Delta_t\mid \cF_{t-1}] &=& (2+\op(1))\eps n, \label{V1} \\
 \sum_{t=1}^{t_1} \Var[\hat\Delta_t\mid \cF_{t-1}] &=& (1+\op(1))\frac{10}{3(r-1)^2}\eps^3n,\label{V2} \\
 \sum_{t=1}^{t_1} \Covar[\hat\Delta_t,\beta_t^{-1}\Delta_t\mid\cF_{t-1}] 
      &=& (1+\op(1))\frac{2}{(r-1)}\eps^2n.\label{V3}
\end{eqnarray}
Moreover, indicating the dependence on $n$ explicitly for a change, the rescaled martingales
\begin{equation}\label{Mdef}
 M_{1,n,t} = (\eps(n) n)^{-1/2} S_{n,t} \hbox{\quad and\quad} 
 M_{2,n,t} = (\eps(n)^3 n)^{-1/2} \hat S_{n,t},
\end{equation}
defined for $0\le t\le t_1(n)$, satisfy the Lindeberg condition \eqref{LC}.
\end{claim}

Assuming the claim for the moment then, rescaling as in \eqref{Mdef}, the bounds \eqref{V1}--\eqref{V3}
give exactly the variance conditions \eqref{VC2} and \eqref{CVC2} of Lemma~\ref{lm2},
with $\sigma_1^2$, $\sigma_2^2$ and $\sigma_{1,2}$ as in \eqref{Vars}.
Thus Lemma~\ref{lm2} implies \eqref{concrete} which,
as noted above, implies Theorem~\ref{thglobal2}.
It remains only to prove the claim. The Lindeberg condition asserts, roughly speaking, that it
is unlikely that any single step in either martingale contributes significantly to the
total variance of the martingale over $t_1$ steps. As in almost all combinatorial settings,
this condition holds with plenty of room to spare. Indeed, the (unrescaled) martingales have step
sizes of order $1$, with strong tail bounds (inherited from the binomial distribution),
and their final variances are much larger than $1$, so the Lindeberg condition holds with plenty
of room to spare. We give a full proof in the Appendix.

It remains to establish \eqref{V1}--\eqref{V3}. This concerns only steps $1,\ldots,t_1$
of our random exploration process, so from now on we only consider $0\le t\le t_1$.
Since we are aiming for convergence in probability, and the event $\cE$ defined in Lemma~\ref{props}
holds whp, much of the time we assume that $\cE$ holds.

By Lemma~\ref{Asmall}(iii), when $\cE$ holds we have
\begin{equation}\label{Amax}
 \max_{t\le t_1}A_t=O(\eps^2 n).
\end{equation}
Since $t\le t_1$, the bound \eqref{bsmall} implies that $\beta_t\sim 1$. Thus,
when $\cE$ holds,
\begin{equation}\label{VDt}
 \Var[\beta_t^{-1}\Delta_t\mid \cF_{t-1}] \sim \Var[\Delta_t\mid \cF_{t-1}] \sim r-1,
\end{equation}
where the final estimate follows from Lemma~\ref{cvar}
and the bound \eqref{Amax} above, recalling that $\eps=o(1)$.
(It also follows from \cite[Eq. (7)]{BR_hyp}, for example.)
Hence, on $\cE$,
\[
 \sum_{t=1}^{t_1} \Var[\beta_t^{-1}\Delta_t\mid \cF_{t-1}] \sim (r-1)t_1 \sim 2\eps n.
\]
Since $\cE$ holds whp, this implies \eqref{V1}.
Next, recalling that $D_t^*=\E[\xi_t\mid \cF_{t-1}]$, we have
\begin{eqnarray*}
 \Var[\Delta_t^*\mid \cF_{t-1}] &=& \Var[\xi_t-D_t^*\mid \cF_{t-1}] \\
&=& \Var[\xi_t\mid \cF_{t-1}] \\
 &=& \E[\xi_t\mid \cF_{t-1}] + \E[\xi_t(\xi_t-1)\mid \cF_{t-1}]  - \E[\xi_t\mid\cF_{t-1}]^2\\
 &=& \E[\xi_t\mid \cF_{t-1}] + O(A_{t-1}^2/n^2+1/n^2),
\end{eqnarray*}
by \eqref{pismall}--\eqref{Exixi}. 
Hence, by \eqref{Amax}, when $\cE$ holds we have
\[
 \Var[\Delta_t^*\mid \cF_{t-1} ] = D_t^* + O(\eps^4).
\]
Now, on $\cE$,
\begin{eqnarray}
 \Covar[\Delta_t^*,\beta_t^{-1}\Delta_t\mid \cF_{t-1}] 
&\sim& \Covar[\Delta_t^*,\Delta_t\mid \cF_{t-1}] \nonumber\\
 &=&  \Covar[\xi_t,\eta_t\mid \cF_{t-1}] \nonumber\\
 &=& \E[\xi_t\eta_t\mid\cF_{t-1}] -\E[\xi_t\mid \cF_{t-1}]\E[\eta_t\mid \cF_{t-1}]\nonumber\\
 &=& O(A_t/n+1/n) = O(\eps^2) = o(\eps),\label{Cov}
\end{eqnarray}
using \eqref{Exieta}, \eqref{Exi}, the bound $\E[\eta_t\mid \cF_{t-1}]=O(1)$ (which
follows from Lemma~\ref{cdist}), and \eqref{Amax}.
Hence
\begin{eqnarray*}
 \Var[\hat\Delta_t\mid \cF_{t-1} ] &=& \Var[\Delta_t^*\mid \cF_{t-1}] + \gamma_t^2\Var[\Delta_t\mid \cF_{t-1}] + \gamma_t\Covar[\Delta_t^*,\Delta_t\mid \cF_{t-1}] \\
 &=& D_t^* + (r-1)(t_1-t)^2/n^2 + o(\eps^2),
\end{eqnarray*}
recalling \eqref{VDt} and \eqref{gt}.
Thus,
\begin{eqnarray}
 \sum_{t=1}^{t_1} \Var[\hat\Delta_t\mid \cF_{t-1}] &=& \sum_{t=1}^{t_1} D_t^* + (r-1)\frac{t_1^3}{3n^2} +o(\eps^3n) \nonumber
\\
 &=& \rho^*n+\frac{8\eps^3}{3(r-1)^2}n + \op(\eps^3n) \nonumber \\
&=& (1+\op(1))\frac{10}{3(r-1)^2}\eps^3n,\nonumber
\end{eqnarray}
by Corollary~\ref{Dscor} and \eqref{rsasymp}. This proves \eqref{V2}.
Finally, since $\beta_t\sim 1$ and $\hat\Delta_t= \gamma_t\Delta_t+\Delta_t^*$, when $\cE$ holds we have
\begin{eqnarray*}
 \Covar[\hat\Delta_t,\beta_t^{-1}\Delta_t\mid\cF_{t-1}] &\sim&
 \Covar[\hat\Delta_t,\Delta_t\mid\cF_{t-1}] \\
 &=& \gamma_t\Var[\Delta_t\mid\cF_{t-1}] + \Covar[\Delta_t^*,\Delta_t\mid \cF_{t-1}]  \\
&=& \frac{t_1-t}{n}(r-1) + o(\eps),
\end{eqnarray*}
from \eqref{gt}, \eqref{VDt} and \eqref{Cov}.
Hence
\begin{equation*}
 \sum_{t=1}^{t_1}  \Covar[\hat\Delta_t,\beta_t^{-1}\Delta_t\mid\cF_{t-1}]
  = (r-1)\frac{t_1^2}{2n}+\op(\eps^2n) =(1+\op(1))\frac{2}{(r-1)}\eps^2n,
\end{equation*} 
establishing \eqref{V3}.
This completes the proof of Claim~\ref{cc} and hence of Theorem~\ref{thglobal2}.
\end{proof}

As we have already remarked, in a follow-up paper~\cite{BRsmoothing}
we prove a local limit version of Theorem~\ref{thglobal2}, using
Theorems~\ref{thglobal2} and~\ref{thsupertail} as tools in the proof. 
This local limit theorem is then used to prove an asymptotic formula for the number of
connected $r$-uniform hypergraphs with a given number of vertices and edges,
in the case where the nullity is small compared to the number of edges.

\appendix
\section{Appendix}

In this appendix we prove the second part of Claim~\ref{cc}, concerning the Lindeberg condition.
We indicate the dependence on $n$ explicitly much of the time,
writing $S_{n,t}$ for $S_t$, and so on. As in the claim, for $n\ge n_0$ and $0\le t\le t_1$ let
\[
 M_{1,n,t} = (\eps n)^{-1/2} S_{n,t} \hbox{\quad and\quad} 
 M_{2,n,t} = (\eps ^3 n)^{-1/2} \hat S_{n,t}.
\]
We must show that these martingales satisfy the Lindeberg condition \eqref{LC}.

\begin{claim}\label{LC2}
The martingale difference sequences
\[
 ( (\eps n)^{-1/2} \Delta_{n,t} )\hbox{\quad and\quad} ((\eps^3 n)^{-1/2} \Delta^*_{n,t} )
\]
satisfy the Lindeberg condition.
\end{claim}

Recall from \eqref{bsmall} that the deterministic quantities $\beta_{n,t}^{-1}$ are bounded.
Multiplying the martingale differences by such bounded factors clearly preserves the Lindeberg
condition. Hence the first part Claim~\ref{LC2} implies the Lindeberg
condition for the martingale $((\eps n)^{-1/2}S_{n,t})$.
For $((\eps^3 n)^{-1/2}\hat S_{n,t})$,
recall from \eqref{hDt} that $\hat\Delta_{n,t} = \gamma_{n,t}\Delta_{n,t}+\Delta^*_{n,t}$.
Thus the relevant differences are
\[
  (\eps^3 n)^{-1/2} \hat\Delta_{n,t} = \eps^{-1}\gamma_{n,t} (\eps n)^{-1/2} \Delta_{n,t}
 + (\eps^3 n)^{-1/2} \Delta^*_{n,t}.
\]
From \eqref{gt} we have $\gamma_{n,t}=O(\eps)$ when $t\le t_1$,
so the deterministic quantities $\eps^{-1}\gamma_{n,t}$ are bounded.
Furthermore, as noted in the proof of Lemma~\ref{lm2},
the Lindeberg condition is preserved by addition.
Thus, the Lindeberg condition for $((\eps^3 n)^{-1/2}\hat S_{n,t})$ follows from Claim~\ref{LC2}.

It remains only to prove Claim~\ref{LC2}. In the calculations, there is plenty of room to spare,
and there are doubtless many other strategies that would work.

Let $\delta>0$ be constant. To establish the Lindeberg condition for the differences
$((\eps n)^{-1/2}\Delta_{n,t})$,
note that if $|(\eps n)^{-1/2}\Delta_{n,t}|\ge \delta$,
then $|\Delta_{n,t}|\ge \delta (\eps n)^{1/2} \ge n^{1/3}$, say, for $n$ large enough.
Since $\eta_{n,t}-\Delta_{n,t}=\E[\eta_{n,t}\mid \cF_{t-1}]=O(1)$ by Lemma~\ref{cdist},
this implies $\eta_{n,t}\ge n^{1/3}-O(1)$,
which has probability $\exp(-\Omega(n^{1/3}))=o(n^{-100})$ by Lemma~\ref{cdist}
and a Chernoff bound. This is more than enough to establish the Lindeberg condition.

For the second part of Claim~\ref{LC2}, recall from \eqref{Dtsdef} that
\[
 D_t^* = \E[\xi_t\mid \cF_{t-1}] \hbox{\quad and\quad} \Delta_{n,t}^* = \xi_t-D_t^*,
\]
where $\xi_t=\xi_{n,t}$ is a random variable taking non-negative integer values.
From \eqref{pismall} and \eqref{Exi} we have the very crude bound
\[
 D_t^* =O(n\pi_t)+O(1/n)=O(1).
\]
Fixing $\delta$, it follows that if $n$ large enough (which we assume from now on)
then we always have $D_t^*< \delta(\eps^3n)^{1/2}$ and thus
$\Delta^*_{n,t} > - \delta(\eps^3n)^{1/2}$. Hence,
\begin{multline*}
 \E[(\Delta^*_{n,t})^2 \ind{|\Delta^*_{n,t}|\ge \delta(\eps^3 n)^{1/2}} \mid \cF_{t-1} ] 
 =  \E[(\Delta^*_{n,t})^2 \ind{\Delta^*_{n,t}\ge \delta(\eps^3 n)^{1/2}} \mid \cF_{t-1}] \\
 \le \E[(\Delta^*_{n,t})^2 \ind{\Delta^*_{n,t}\ge 2} \mid \cF_{t-1} ].
\end{multline*}
When $\Delta^*_{n,t}\ge 2$ then, since $2\le \Delta^*_{n,t} \le \Delta^*_{n,t}+D_t^* = \xi_t$, we have
\[
 (\Delta^*_{n,t})^2 \le \xi_t^2 \le 2\xi_t(\xi_t-1).
\]
From \eqref{Exixi} and Lemma~\ref{Asmall}(iii), on the event $\cE$ we have $\E[\xi_t(\xi_t-1)\mid\cF_{t-1}]=O(\eps^4)$,
uniformly in $t\le t_1$. Hence, on $\cE$,
\[
 \E[(\Delta^*_{n,t})^2 \ind{|\Delta^*_{n,t}|\ge \delta(\eps^3 n)^{1/2}} \mid \cF_{t-1} ] =O(\eps^4) = o(\eps^2).
\]
Summing over $t\le t_1$ and using Markov's inequality, we see that
\[
 \sum_{t\le t_1} \E[(\Delta^*_{n,t})^2 \ind{|\Delta^*_{n,t}|\ge \delta(\eps^3 n)^{1/2}} \mid \cF_{t-1} ]
 =\op(\eps^3n),
\]
which is exactly the Lindeberg condition for $((\eps^3 n)^{-1/2} \Delta^*_{n,t} )$.

\medskip
\noindent
{\bf Acknowledgements.}
We would like to thank the referees for suggestions leading
to improvements in the presentation, and Svante Janson for pointing us towards~\cite{BrownEagleson}.
We would also like to apologise to the editors for the very long time taken to
revise the paper.

\end{document}